\documentclass[a4paper,11pt,reqno]{amsart}

\usepackage{hyperref}
\usepackage{amsthm,latexsym,amssymb,amsmath,stmaryrd}
\usepackage{pifont}
\usepackage{times,mathptmx}
\usepackage{enumerate}
\usepackage{pstricks}
\usepackage{auto-pst-pdf}

\allowdisplaybreaks
\newcommand{\dS}{\mathrm{d}\Sigma}

\newcommand{\ddt}{\frac{d}{dt}}

\newcommand{\nSt}{\overline{\nabla}}
\newcommand{\Db}{\overline{D}}
\newcommand{\Zb}{\overline{Z}}
\newcommand{\Fb}{\overline{F}}
\newcommand{\FE}{\mathcal{F\hspace{-0.9mm}E}_{\hspace{-0.7mm}sc}}
\renewcommand{\tt}{\tilde{t}}
\newcommand{\tS}{\tilde{\Sigma}}
\newcommand{\loc}{_\mathrm{loc}}
\newcommand{\locsc}{_{\mathrm{loc},sc}}
\newcommand{\locK}{_{\mathrm{loc},K}}
\newcommand{\locJK}{_{\mathrm{loc},J(K)}}
\newcommand{\Solve}{\mathrm{Solve}}

\newcommand{\D}{\mathcal{D}}

\newcommand{\N}{\mathbb{N}}
\newcommand{\R}{\mathbb{R}}
\def\C{\mathbb{C}}
\newcommand{\K}{\mathbb{K}}
\newcommand{\<}{\langle}
\renewcommand{\>}{\rangle}
\renewcommand{\phi}{\varphi}
\newcommand{\na}{\nabla}
\newcommand{\id}{\,\mathrm{id}}
\newcommand{\tr}{\mathrm{tr}}
\newcommand{\eps}{\varepsilon}
\newcommand{\dV}{\,\mathrm{dV}}
\newcommand{\vol}{\,\mathrm{vol}}
\newcommand{\AL}{\mathrm{A}_L}
\newcommand{\supp}{\mathrm{supp}}
\renewcommand{\Re}{\mathrm{Re}}

\newtheorem{thm}{Theorem}

\newtheorem{cor}[thm]{Corollary}

\newtheorem{lem}[thm]{Lemma}

\theoremstyle{definition}
\newtheorem{ex}[thm]{Example}
\newtheorem{rem}[thm]{Remark}

\newtheorem{dfn}[thm]{Definition}

\newcommand{\tref}[1]{Theorem~\ref{#1}}

\newcommand{\lref}[1]{Lemma~\ref{#1}}
\newcommand{\cref}[1]{Corollary~\ref{#1}}

\title{Initial value problems for wave equations on manifolds}
\author{Christian B\"ar}
\address{Universit\"at Potsdam, Institut f\"ur Mathematik, Am Neuen Palais 10, 14469 Potsdam, Germany}
\email{baer@math.uni-potsdam.de}
\urladdr{http://geometrie.math.uni-potsdam.de/}
\author{Roger Tagne Wafo}
\address{University of Douala, Faculty of Science, Dept.\ of Mathematics and Computer Science, P.O.Box 24157 Douala, Cameroon}
\email{rtagnewafo@yahoo.com}

\date{\today}
\keywords{wave equation, globally hyperbolic Lorentz manifold, Cauchy problem, Goursat problem, finite energy sections}
\subjclass[2010]{35L05, 35L15, 58J45}

\begin{document}

\begin{abstract}
We study the global theory of linear wave equations for sections of vector bundles over globally hyperbolic Lorentz manifolds.
We introduce spaces of finite energy sections and show well-posedness of the Cauchy problem in those spaces.
These spaces depend in general on the choice of a time function but it turns out that certain spaces of finite energy \emph{solutions} are independent of this choice and hence invariantly defined.

We also show existence and uniqueness of solutions for the Goursat problem where one prescribes initial data on a characteristic partial Cauchy hypersurface.
This extends classical results due to H\"ormander.
\end{abstract}

\maketitle


\section*{Introduction}

This paper is concerned with the \emph{global} theory of initial value problems for linear wave equations on curved spacetimes.
Applications are numerous:
electromagnetic radiation and gravitational waves in general relativity are described by such equations, the Klein-Gordon equation from quantum field theory and many equations from linear relativistic elasticity theory fall in this category, just to name a few.
The study of nonlinear wave equations like the Einstein equations also requires a good understanding of the linear theory.

Wave equations form a classical topic in the theory of partial differential equations.
Traditionally they are studied on subsets of Minkowski space.
There are excellent expositions of this theory in textbook format such as \cite{A09,H97,SS} and many more.
This theory can be used to understand the \emph{local} theory of wave equations on manifolds as well, see e.g.\ \cite{F,G}.

The setup in the present paper is the following:
The underlying spacetime $M$ on which the waves are defined is a Lorentz manifold.
The manifold $M$ may have any dimension.
In order to be able to set up a reasonable initial value problem we have to assume that the Lorentz manifold is globally hyperbolic.
This is a geometric condition which can be formulated in various seemingly different but equivalent ways.
One of them would be the existence of a Cauchy hypersurface, another one the existence of suitable time functions, so-called Cauchy temporal functions.
The waves are modeled by sections of a vector bundle.
So we allow for vector-valued functions and hence for systems of partial differential equations.
The equation is given by a second-order linear differential operator $P$ whose principal symbol is given by the Lorentz metric.

We are interested in solutions of the equation $Pu=f$ with given $f$ where $u$ should be defined on all of $M$.
We consider two types of initial value problems, commonly known as the Cauchy and the Goursat problem.
For the Cauchy problem we fix a spacelike Cauchy hypersurface $\Sigma$ and prescribe $u$ and its normal derivative along $\Sigma$.
For the Goursat problem we fix a characteristic (lightlike) partial Cauchy hypersurface and prescribe only $u$ along $\Sigma$.

We show well-posedness of the Cauchy problem in suitable function spaces (\tref{thm:Cauchy1}).
The initial data along $\Sigma$ lie in certain Sobolev spaces and $f$ is assumed locally square integrable in time and of some Sobolev regularity in space.
It then turns out that the solution $u$ lies in a space of finite energy functions meaning that $u$ and its time derivative are continuous in time and of some Sobolev regularity in space.
The definition of these function spaces requires a splitting of the spacetime in space and time for which there is no canonical choice.
In general, these function spaces do indeed depend on the choice of this splitting.
It will turn out however that the space of finite energy solutions to the homogeneous Cauchy problem, i.e.\ $f=0$, is independent of the choice of time function (\cref{cor:FEkerPIndep}).
As to the inhomogeneous problem, there is one particular Sobolev regularity scale in space for which the solution space is also independent of the time function (\cref{cor:FEPIndep}).

For the Goursat problem, there are a number of existence and uniqueness results for very special Cauchy hypersurfaces (also for quasi-linear equations) such as \cite{Ca82,CBCMG11,Do02} for characteristic cones and \cite{CP12,R90} for the intersection of two charactistic hyperplanes.
The characteristic initial value problem has been used to construct solutions to the coupled Einstein-Maxwell equations which develop a black hole in the future but have complete past \cite{Daf09} and to construct Hadamard states for quantum field theory on curved spacetimes \cite{GW14}.

We allow for arbitrary characteristic partial Cauchy hypersurfaces $\Sigma$ and show existence and uniqueness of solutions in the future $J^+(\Sigma)$ of the characteristic partial Cauchy hypersurface provided $J^+(\Sigma)$ is past compact (Theorem~\ref{thm:CharacteristicUniqueness}).
Without this geometric assumption existence of solutions still holds (Theorem~\ref{thm:GeneralExistence}) but uniqueness fails as is easily seen by examples.

The paper is organized as follows:
In the section on preliminaries we first recall a few basic notions from Lorentz geometry.
Then we introduce various spaces of sections of a vector bundle: smooth sections, distributional sections, square integrable and Sobolev sections and finite energy sections.
We describe the appropriate topologies on these section spaces.
Finally we recall the notion of a wave operator and give a few examples.

The second section is the analytic core of the paper.
We prove the energy estimate (\tref{thm:enest}, see also \cref{cor:EnEstOptimal}) which is behind the well-posedness of the Cauchy problem.

In the third section we derive the well-posedness of the Cauchy problem.
This complements the results in \cite{BGP07} where different methods were used to show the well-posedness of the Cauchy problem for smooth sections.
On the side, we see that smooth sections are dense in finite energy sections and smooth solutions are dense in finite energy solutions (\cref{cor:dense}).
Moreover, finite energy solutions to the homogeneous Cauchy problem are shown to have appropriate Sobolev regularity for all Sobolev scales.
For the solutions to the inhomogeneous Cauchy problem this still holds for one particular Sobolev scale (\cref{cor:FEHk}).

In the fifth section we use the results on the Cauchy problem to show existence of solutions to the Goursat problem.
Uniqueness is based on a Green's formula which we prove in the appendix.
These results generalize H\"ormander's Theorem~2 in \cite{H} where he shows existence and uniqueness under the assumption that $M$ is spatially compact, $f=0$ and $\beta\equiv 1$ where $\beta$ is the function from \eqref{eq:gbeta}.

In \cite{N1} H\"ormander's result has been shown to hold under rather weak assumptions on the regularity of the metric and the coefficients of the operator.
We have made no attempt to minimize the regularity assumptions of the geometric data.
The Lorentz metric and the coefficients of the differential operator are assumed to be smooth.
The spacelike Cauchy hypersurface on which the initial data for the Cauchy problem are prescribed is also assumed to be smooth.
This has the advantage that initial data of arbitrary Sobolev regularity can be treated.
On the other hand, for the characteristic Cauchy hypersurface occuring in the Goursat problem we make no regularity assumptions at all. 
A smoothness assumption would exclude basically all interesting examples.
As a consequence, we can only consider initial data of one particular Sobolev regularity.

\emph{Acknowledgments.}
It is a pleasure to thank Piotr Chrusciel, Miguel S\'anchez, Elmar Schrohe, and Christoph Stephan for very helpful discussion.
The first author thanks \emph{Sonderforschungsbereich 647} funded by \emph{Deutsche Forschungsgemeinschaft} for financial support.
The second author likes to thanks the \emph{Einstein Foundation Berlin} for financial support and the \emph{University of Potsdam} for its hospitality.


\section{Preliminaries}

In this section we collect the necessary background material on globally hyperbolic manifolds, on various section spaces and on wave operators.
We use the convention $\N=\{1,2,3,\ldots\}$ and $\N_0=\N \cup \{0\}$.

\subsection{Globally hyperbolic manifolds}
We summarize various facts about globally hyperbolic Lorentzian manifolds. 
For details the reader is referred to one of the classical textbooks \cite{BEE96,HE73,ON83}.
Throughout  this article, $M$ will denote a timeoriented Lorentzian manifold.
We use the convention that the signature of $M$ is $(-+\cdots+)$.
Note that we do not specify the dimension of $M$ nor do we assume orientability or connectedness.

A subset $\Sigma\subset M$ is called a \emph{Cauchy hypersurface} if every inextensible timelike curve in $M$ meets $\Sigma$ exactly once.
Any Cauchy hypersurface is a Lipschitz hypersurface of $M$.
All Cauchy hypersurfaces of $M$ are homeomorphic.

If a timeoriented Lorentzian manifold $M$ possesses a Cauchy hypersurface then $M$ is called \emph{globally hyperbolic}.
This class of Lorentzian manifolds contains many important examples:
Minkowski space, Friedmann models, the Schwarzschild model and deSitter spacetime are globally hyperbolic.

Bernal and S\'anchez proved an important structural result \cite[Thm.~1.1]{BS05}:
Any globally hyperbolic Lorentzian manifold has a \emph{Cauchy temporal function}.
This is a smooth function $t:M\to\R$ with past-directed timelike gradient $\na t$ such that the levels $t^{-1}(s)=:\Sigma_s$ are (smooth spacelike) Cauchy hypersurfaces if nonempty.
The Lorentzian metric of $M$ then takes the form 
\begin{equation}
g = -\beta dt^2 + g_t
\label{eq:gbeta}
\end{equation}
where $\beta$ is a positive smooth function on $M$ and $g_s$ denotes a Riemannian metric on $\Sigma_s$ depending smoothly on the parameter $s\in t(M)$.

From now on let $M$ always be globally hyperbolic.
For any $x\in M$ we denote by $J^+(x)$ the set all points that can be reached by future-directed causal curves emanating from $x$.
For any subset $A\subset M$ we put $J^+(A):=\bigcup_{x\in A}J^+(x)$.
If $A$ is compact, then $J^+(A)$ is closed.

We denote by $I^+(x)$ the set of all points in $M$ that can be reached by future-directed timelike curves emanating from $x$.
The set $I^+(x)$ is the interior of $J^+(x)$; in particular, it is an open subset of $M$.
For any subset $A\subset M$ the union $I^+(A):=\bigcup_{x\in A}I^+(x)$ is also open.

Interchanging the roles of future and past, we similarly define $J^-(x)$, $J^-(A)$, $I^-(x)$, and $I^-(A)$. 
Furthermore, we set $J(A):=J^+(A)\cup J^-(A)$.
A subset $A\subset M$ is called \emph{spatially compact} if $A$ is closed and there exists a compact subset $K\subset M$ with $A\subset J(K)$.
The intersection of any spatially compact subset and any Cauchy hypersurface is compact.

A closed subset $A\subset M$ is called \emph{past compact} if $A\cap J^-(x)$ is compact for every $x\in M$.
It then follows that $A\cap J^-(K)$ is compact for all compact subsets $K\subset M$.

A Lorentzian manifold $M$ is globally hyperbolic if and only if $J^+(x)\cap J^-(y)$ is compact for all $x,y\in M$ and there are no causal loops \cite{BS07}.
This is convenient if one wants to check that an open subset $N$ of a globally hyperbolic manifold $M$ is itself globally hyperbolic.
One only needs to check that for any $x,y\in N$ the set $J^+(x)\cap J^-(y)$ is contained in $N$. 
For instance, if $A\subset M$ is any subset of a globally hyperbolic manifold, then $I^+(A)$ and $I^-(A)$ are also globally hyperbolic.

\subsection{Partial Cauchy hypersurfaces}
We will need a relaxation of the concept of Cauchy hypersurfaces in order to properly formulate characteristic initial value problems.
A subset $\Sigma\subset M$ is called \emph{achronal} if every inextensible timelike curve in $M$ meets $\Sigma$ at most once (rather than exactly once as for Cauchy hypersurfaces).
A closed achronal subset $\Sigma\subset M$ which is a topological hypersurface will be called a \emph{partial Cauchy hypersurface}.
Every partial Cauchy hypersurface is a Lipschitz hypersurface, see \cite[Prop.~14.25]{ON83} and its proof.

Every Cauchy hypersurface is a partial Cauchy hypersurface.
If $A$ is a future set, i.e.\ $I^+(A)\subset A$, then its boundary $\Sigma=\partial A$ is a partial Cauchy hypersurface \cite[p.~415]{ON83}.
Typical examples are lightcones $\Sigma =  \partial J^+(p)$ or, more generally, $\Sigma=\partial J^+(B)$ for any subset $B\subset M$.
A similar remark applies to past sets.

For any partial Cauchy hypersurface $\Sigma$, the three sets $\Sigma$, $I^+(\Sigma)$ and $I^-(\Sigma)$ are mutually disjoint (because of achronality) and the union $I^+(\Sigma)\cup \Sigma \cup I^-(\Sigma)$ is an open subset of $M$ (because $\Sigma$ is a hypersurface).
Moreover, $J^\pm(\Sigma) = \Sigma\cup I^\pm(\Sigma)$.

Note that the definition of partial Cauchy hypersurfaces in \cite[p.~204]{HE73} is more restrictive than ours;
Hawking and Ellis demand that $\Sigma$ be acausal rather than achronal.
This would exclude lightlike partial Cauchy hypersurfaces which are precisely the ones we will be interested in.

\subsection{Smooth sections of vector bundles}
Let $S\to M$ be a (real or complex) vector bundle over $M$.
We denote the space of smooth sections of $S$ by $C^\infty(M;S)$ or briefly by $C^\infty(M)$ if the choice of $S$ is clear from the context.
Any connection $\na$ on $S$ induces, together with the Levi-Civita connection on $T^*M$, a connection on $T^*M^{\otimes \ell}\otimes S$ for any $\ell\in\N_0$.
For any $f\in C^\infty(M;S)$, the $\ell^\mathrm{th}$ covariant derivative $\na^{\ell}f := \na\cdots\na\na f$ is a smooth section of  $T^*M^{\otimes \ell}\otimes S$.

For any compact subset $K\subset M$, any $m\in\N_0$, any connection $\na$ on $S$ and any auxiliary norms $|\cdot |$ on $T^*M^{\otimes \ell}\otimes S$ we define the seminorm
\[
\|f\|_{K,m,\na,|\cdot |} := \max_{\ell=0,\cdots,m}\,\,\max_{x\in K} |\na^{\ell}f(x)|
\]
for $f\in C^\infty(M;S)$.
By compactness of $K$, different choices of $\na$ and $|\cdot |$ lead to equivalent seminorms.
For this reason, we may suppress $\na$ and $|\cdot |$ in the notation and write $\|f\|_{K,m}$ instead of $\|f\|_{K,m,\na,|\cdot |}$.
This family of seminorms is separating and turns $C^\infty(M;S)$ into a locally convex topological vector space.
If we choose a sequence $K_1 \subset K_2 \subset K_3 \subset \cdots\subset M$ of compact subsets with $\bigcup_{i=1}^\infty K_i =M$ and such that each $K_i$ is contained in the interior of $K_{i+1}$, then the countable subfamily $\|\cdot\|_{K_i,i}$ of seminorms is equivalent to the original family.
Hence $C^\infty(M;S)$ is metrizable.
An Arzel\'a-Ascoli argument shows that  $C^\infty(M;S)$ is complete.
Thus  $C^\infty(M;S)$ is a Fr\'echet space.
A sequence of sections converges in $C^\infty(M;S)$ if and only if the sections and all their (higher) derivatives converge locally uniformly.

For any closed subset $A\subset M$, we equip 
\[
C^\infty_A(M;S) := \{u\in C^\infty(M;S) \mid \supp(u)\subset A\}
\]
with the relative topology.
The space of all compactly supported smooth sections 
\[
C^\infty_c(M;S) := \bigcup_{K\subset M \atop \mathrm{compact}} C^\infty_K(M;S) 
\]
is equipped with the strict inductive limit topology.
This turns $C^\infty_c(\Sigma;S)$ into a locally convex topological vector space.
The inclusion maps $C^\infty_K(\Sigma;S) \hookrightarrow C^\infty_c(\Sigma;S)$ are continuous.
For any locally convex topological vector space $X$, a linear map $\Lambda:C^\infty_c(\Sigma;S)\to X$ is continuous if and only if the restriction of $\Lambda$ to any subspace $C^\infty_K(\Sigma;S)$ is continuous.
A sequence of sections converges in $C^\infty_c(M;S)$ if and only if the supports of the sections are contained in a common compact subset of $M$ and the sequence converges in $C^\infty(M;S)$.

In the same manner, we equip the space
\[
C^\infty_{sc}(M;S) := \bigcup_{A\subset M \atop \mathrm{spatially\,\, compact}} C^\infty_A(M;S) 
\]
of sections with spatially compact support with the strict inductive limit topology.

\subsection{Distributional sections}

Let $\dV$ be the volume element induced by the Lorentzian metric on $M$.
Again, let $S\to M$ be a real or complex vector bundle over $M$.
In the first case we write $\K=\R$ and in the latter case $\K=\C$.
Let $S^*\to M$ be the dual bundle, i.e., the fibers $S^*_x$ are the $\K$-dual spaces of the fibers $S_x$.

Compactly supported smooth sections of $S^*$ are called \emph{test sections} for $S$.
We denote by $\D'(M;S)$ (or briefly $\D'(M)$ if $S$ is clear from the context) the space of all continuous $\K$-linear functionals on $C^\infty_c(M;S^*)$ and call it the space of \emph{distributional sections} of $S$.
The evaluation of a distributional section $u$ on a test section $\phi$ will be denoted by $u[\phi]$.
Any locally integrable section $u$ can be considered as a distributional section by $u[\phi] = \int_M \phi(x)(u(x))\dV(x)$.
Here $\phi(x)(u(x))\in\K$ is the number obtained by evaluating the linear form $\phi(x)\in S^*_x$ on the element $u(x)\in S_x$.

\subsubsection{The topology on $\D'(M;S)$}
We provide $\D'(M;S)$ with the weak*-topology induced by the topology of $C^\infty_c(M;S^*)$.
Hence a sequence $(u_j)$ in $\D'(M;S)$ converges if and only if $u_j[\phi]$ converges for every test section $\phi\in C^\infty_c(M;S^*)$.

\subsubsection{The formally dual operator}
For any linear differential operator $P: C^\infty(M;S) \to  C^\infty(M;S)$ there is a unique \emph{formally dual operator} $P^\dagger: C^\infty(M;S^*) \to  C^\infty(M;S^*)$ of the same order, characterized by 
$$
\int_M \phi(Pu) \dV = \int_M (P^\dagger\phi)(u) \dV 
$$
for all $u\in C^\infty(M;S)$ and $\phi\in C^\infty(M;S^*)$ with $\supp(\phi) \cap \supp(u)$ compact.

\subsubsection{Extension to distributional sections}
The adjoint operator of $P^\dagger:C^\infty_c(M;S^*)\to C^\infty_c(M;S^*)$ extends $P$ to distributional sections.
In other words, the extension $P:\D'(M;S)\to\D'(M;S)$ is given by
\[
Pu[\phi] = u[P^\dagger\phi]
\]
where $u\in \D'(M;S)$ and $\phi\in C^\infty_c(M;S^*)$.
Any linear differential operator $P$ is continuous as an operator $C^\infty(M;S)\to C^\infty(M;S)$, as an operator $C^\infty_c(M;S)\to C^\infty_c(M;S)$, as an operator $C^\infty_{sc}(M;S)\to C^\infty_{sc}(M;S)$ and as an operator $\D'(M;S)\to\D'(M;S)$.

\subsection{Square integrable sections}

Now assume that the vector bundle $S$ comes equipped with a Riemannian or Hermitian metric $\<\cdot,\cdot\>$, antilinear in the first argument and linear in the second.
For $u,v\in C^\infty_c(M;S)$ we define the $L^2$-scalar product by 
\[
(u,v)_{L^2(M)}
:=
\int_M \<u(x),v(x)\> \dV(x) \,
\]
and the $L^2$-norm by
\[
\|u\|_{L^2(M)}
:=
\sqrt{(u,u)_{L^2(M)}} \, .
\]
The completion of $C^\infty_c(M;S)$ with respect to the $L^2$-norm will be denoted by $L^2(M;S)$.

\subsubsection{The formally adjoint operator}
For any linear differential operator $P: C^\infty(M;S) \to  C^\infty(M;S)$ there is a unique \emph{formally adjoint operator} $P^*: C^\infty(M;S) \to  C^\infty(M;S)$ of the same order, characterized by 
$$
(\phi,P\psi)_{L^2(M)} = (P^*\phi,\psi)_{L^2(M)}
$$
for all $\phi,\psi\in C^\infty(M;S)$ with $\supp(\phi) \cap \supp(\psi)$ compact.

\subsection{Sobolev spaces}
\label{subsec:Sobolev}
We introduce Sobolev spaces of sections on a manifold in a manner which will be convenient later.

\subsubsection{Compact manifolds}
We do it on compact manifolds first.
Let $\Sigma$ be a compact manifold without boundary.
Let $S\to \Sigma$ be a real or complex vector bundle.
We equip $\Sigma$ with an auxiliary Riemannian metric and $S$ with a Riemannian or Hermitian metric and a compatible connection $\nSt$.

Denote the formal adjoint of $\nSt:C^\infty(\Sigma;S)\to C^\infty(\Sigma,T^*\Sigma\otimes S)$ by  $\nSt^*: \mbox{$C^\infty(\Sigma,T^*\Sigma\otimes S)$} \to C^\infty(\Sigma;S)$.
The Laplace-type operator $\nSt^*\nSt+\id: C^\infty(\Sigma;S) \to C^\infty(\Sigma;S)$ is elliptic, positive and essentially selfadjoint in $L^2(\Sigma;S)$.
We denote the square root of the selfadjoint extension of $\nSt^*\nSt+\id$ by $\Db$.
Then we have for all $k\in\R$:
\begin{equation}
\Db^k = \big(\nSt^*\nSt+\id\big)^{k/2} .
\label{eq:Dk}
\end{equation}
For each $k\in\R$, $\Db^k$ is a positive, elliptic, selfadjoint classical pseudo-differential operator of order $k$.

We define the $k^\mathrm{th}$ \emph{Sobolev norm} of $u\in C^\infty(\Sigma;S)$ by
\[
\|u\|_{H^k(\Sigma)} := \|\Db^k u\|_{L^2(\Sigma)}
\]
and the \emph{Sobolev space} $H^k(\Sigma;S)$ as the completion of $C^\infty(\Sigma;S)$ with respect to \mbox{$\|\cdot\|_{H^k(\Sigma)}$}.
For $k=0$ we have $\|\cdot\|_{H^0(\Sigma)}=\|\cdot\|_{L^2(\Sigma)}$ and for $k=1$ we get
\begin{align*}
\|u\|_{H^1(\Sigma)}^2
&=
(\Db u,\Db u)_{L^2(\Sigma)}
=
(\Db^2 u,u)_{L^2(\Sigma)} \\
&=
((\nSt^*\nSt+\id) u,u)_{L^2(\Sigma)}
=
\|\nSt u\|_{L^2(\Sigma)}^2 + \|u\|_{L^2(\Sigma)}^2 \, .
\end{align*}
More generally, for $k\in\N_0$, a norm equivalent to $\|\cdot\|_{H^k(\Sigma)}$ is given by 
\[
\|u\|^2 = \sum_{\ell=0}^k \|\nSt^\ell u\|^2_{L^2(\Sigma)} \, .
\]
Different choices of metrics and connection lead to equivalent Sobolev norms and hence to the same Sobolev spaces.

\subsubsection{Noncompact manifolds}
Now we drop the assumption that $\Sigma$ is compact.
Let $K\subset \Sigma$ be a compact subset.
We want to define the space of Sobolev sections whose support is contained in $K$.

We choose a compact subset $K_1\subset\Sigma$ such that the interior of $K_1$ contains $K$ and the boundary $\partial K_1$ is smooth.
Now let $\Sigma'$ be the double of $K_1$ as a differentiable manifold.
In other words, $\Sigma' = K_1 \cup_{\partial K_1} K_2$ where $K_2$ is another copy of $K_1$ and the two copies are glued along their boundary $\partial K_1=\partial K_2$.

\begin{center}
\psset{unit=0.05Em}
\begin{pspicture}(100,400)(675,710)
\pscustom[linewidth=0.7,fillcolor=lightgray,fillstyle=solid] 
{\newpath
\moveto(345.8,650.9)                              
\curveto(347.3,650.8)(349.2,650.8)(351.8,650.9)
\curveto(358.1,651.2)(368.7,653.5)(377.6,654.6)
\curveto(395.4,656.7)(414.7,659.7)(420.8,660.8)
\curveto(433.1,663.2)(450,668.1)(463.5,672.3)      
\curveto(461.3,648.9)(448.3,648.2)(448.5,647.8)    
\curveto(444.2,624.7)(458.5,610)(458.5,610)
\curveto(449.4,615.1)(442.5,618.7)(435.5,621.9)   
\curveto(428.6,625.1)(421.5,627.9)(415.4,629.3)
\curveto(412.3,630)(395.1,632.9)(378.3,635.4)
\curveto(369.9,636.6)(361.6,637.7)(355.2,638.4)
\curveto(352,638.7)(349.2,639)(347.2,639)
\moveto(345.8,650.9)                             
\closepath 
}
\pscustom[linewidth=0.7,fillcolor=darkgray,fillstyle=solid] 
{\newpath
\moveto(463.5,672.3) 
\curveto(487.4,678.9)(507.8,684.8)(527.2,689)      
\curveto(546.5,693.3)(564.6,695.9)(578.8,694.7)
\curveto(634.9,689.9)(672.9,661.2)(662.8,630.6)
\curveto(655.7,609.1)(626.7,592.3)(590.5,585.9)
\curveto(575.3,583.2)(558.8,582.4)(542.1,583.8)
\curveto(528,585)(513.9,587.3)(501.1,590.7)
\curveto(488.3,594.1)(476.8,598.7)(467.6,604.5)
\curveto(465.4,606)(460.8,608.7)(458.5,610)               
\curveto(458.5,610)(444.2,624.7)(448.5,647.8)      
\curveto(448.3,648.2)(461.3,648.9)(463.5,672.3) 
\closepath                                        
}
\pscustom[linewidth=0.7]                                     
{\newpath
\moveto(379.5,654.5)
\curveto(377.5,654)(376,649.3)(376.2,644)
\curveto(376.4,639.3)(378,636)(380,635.5)
}
\pscustom[linewidth=0.7,linestyle=dashed,dash=1.2 1.7]       
{\newpath
\moveto(379.7,635.5)
\curveto(381.7,635.3)(383.4,639.4)(383.5,644.7)
\curveto(383.6,649.9)(381,654.4)(379,654.5)
\curveto(379.9,654.5)(379.8,654.5)(379.8,654.5)
}
\pscustom[linewidth=0.7,fillcolor=white,fillstyle=solid]
{\newpath
\moveto(556.2,637)
\curveto(563.9,645)(573.2,647.4)(581.9,647.3)
\curveto(590.5,647.2)(599.9,644.7)(606.2,637.5)
}
\psellipticarc[linewidth=0.7](580.5,647.25)(35,15){-179}{0}
\psellipticarc[linewidth=0.7,fillcolor=white,fillstyle=solid](580.5,647.25)(35,15){202}{340}
\psline[linecolor=white,linewidth=0.7](558,637)(604,637.5)

\uput{1}[0](410,644){\psframebox*[framearc=.3]{$K_1$}}                              
\uput{1}[0](510,640){\psframebox*[framearc=.3]{$K$}}
\uput{1}[45](660,675){$\Sigma$}

\pscustom[linewidth=0.7,fillcolor=lightgray,fillstyle=solid]
{\newpath
\moveto(379,463.4)                               
\curveto(364.3,463.6)(349.9,462)(347.3,461.3) 
\curveto(342.2,459.9)(336.3,457.1)(330.4,453.9)
\curveto(324.6,450.7)(318.8,447.1)(314,443.9)
\curveto(309.2,440.7)(305.4,438)(303.5,436.5)
\curveto(295.9,430.7)(286.2,426.1)(275.5,422.7)
\curveto(264.7,419.3)(253,417)(241.2,415.8)
\curveto(227.2,414.4)(213.3,415.2)(200.6,417.9)
\curveto(170.3,424.3)(146,441.1)(140,462.6)
\curveto(131.6,493.2)(163.4,521.9)(210.6,526.7)
\curveto(222.2,527.9)(237.5,525.3)(253.7,521)
\curveto(269.8,516.8)(287,510.9)(302.6,505.5)
\curveto(318.3,500.1)(332.5,495.2)(342.7,492.8)
\curveto(347.9,491.7)(364.1,489.5)(379,490.6)      
\curveto(395.4,489.9)(414.7,491.7)(420.8,492.8)  
\curveto(433.1,495.2)(450,500.1)(463.5,504.3)
\curveto(461.3,480.9)(448.3,480.2)(448.5,479.8)  
\curveto(444.2,456.7)(458.5,442)(458.5,442)
\curveto(449.4,447.1)(442.5,450.7)(435.5,453.9)  
\curveto(428.6,457.1)(421.5,459.9)(415.4,461.3)
\curveto(412.3,462)(395.1,463.7)(379,463.4)
}
\pscustom[linewidth=0.7,fillcolor=white,fillstyle=solid]
{\newpath
\moveto(182.4,471.5)
\curveto(187.7,478.7)(195.6,481.2)(202.8,481.3)
\curveto(210,481.4)(217.9,479)(224.3,471)
}
\psellipticarc[linewidth=0.7](204,481.5)(30,15){-179}{0}
\psellipticarc[linewidth=0.7,fillcolor=white,fillstyle=solid](204,481.5)(30,15){204}{333}  
\psline[linecolor=white,linewidth=0.7](184,471.5)(222,471)

\psellipticarc[linewidth=0.7,linestyle=dashed,dash=1.2 1.7](379,477)(5,13.5){-90}{90}   
\psellipticarc[linewidth=0.7](379,477)(5,13.5){90}{-90}                                 

\pscustom[linewidth=0.7,fillcolor=darkgray,fillstyle=solid]
{\newpath 
\moveto(463.5,504.3)                             
\curveto(487.4,510.9)(507.8,516.8)(527.2,521)   
\curveto(546.5,525.3)(564.6,527.9)(578.8,526.7) 
\curveto(634.9,521.9)(672.9,493.2)(662.8,462.6) 
\curveto(655.7,441.1)(626.7,424.3)(590.5,417.9) 
\curveto(575.3,415.2)(558.8,414.4)(542.1,415.8) 
\curveto(528,417)(513.9,419.3)(501.1,422.7)     
\curveto(488.3,426.1)(476.8,430.7)(467.6,436.5) 
\curveto(465.4,438)(460.8,440.7)(458.5,442)
\curveto(458.5,442)(444.2,456.7)(448.5,479.8)    
\curveto(448.3,480.2)(461.3,480.9)(463.5,504.3)
\closepath 
}
\pscustom[linewidth=0.7,fillcolor=white,fillstyle=solid]
{\newpath
\moveto(556.2,469) 
\curveto(563.9,477)(573.2,479.4)(581.9,479.3)
\curveto(590.5,479.2)(599.9,476.7)(606.2,469.5) 
}

\psellipticarc[linewidth=0.7](580.5,479.25)(35,15){-179}{0}  
\psellipticarc[linewidth=0.7,fillcolor=white,fillstyle=solid](580.5,479.25)(35,15){202}{340}  
\psline[linecolor=white,linewidth=0.7](558,469)(604,469.5)

\uput{1}[0](510,472){\psframebox*[framearc=.3]{$K$}}                        
\uput{1}[0](270,472){\psframebox*[framearc=.3]{$K_2$}}                        
\uput{1}[45](660,507){$\Sigma'$}

\end{pspicture}

\emph{Fig.~1}
\end{center}

Similarly, we double the restriction of the bundle $S$ to $K_1$ and obtain a bundle $S'\to\Sigma'$.
We extend the given metrics and connection on $K$, considered as a subset $K\subset K_1 \subset \Sigma'$, to smooth metrics and connection on $\Sigma'$.
Now we can consider any smooth section $u$ of $S$ over $\Sigma$ whose support is contained in $K$ also as a smooth section of $S'$ over $\Sigma'$.
We put
\[
\|u\|_{H^k(K)} := \|u\|_{H^k(\Sigma')} 
\]
and define $H^k_K(\Sigma;S)$ as the completion of $C^\infty_K(\Sigma;S)$ with respect to this norm.
Again, different choices (of $K_1$, metrics, connection) lead to equivalent norms and to the same Sobolev spaces.
If $K\subset K'$, then the inclusion $C^\infty_K(\Sigma;S) \subset C^\infty_{K'}(\Sigma;S)$ induces a continuous linear ``inclusion''  map $H^k_K(\Sigma;S) \hookrightarrow H^k_{K'}(\Sigma;S)$.

\subsubsection{Sobolev sections with compact support}\label{sss:SobCompSupp}
Next we define Sobolev spaces of sections with compact support, but without fixing the support.
We put\footnote{Strictly speaking, $H^k_c(\Sigma;S)$ is the direct limit of the direct system given by $\{H^k_K(\Sigma;S)\}_{K}$ and the inclusion maps.}
\[
H^k_c(\Sigma;S) := \bigcup_{K\subset\Sigma \atop \mathrm{compact}} H^k_K(\Sigma;S) \, .
\]
As for smooth sections, we equip $H^k_c(\Sigma;S)$ with the strict inductive limit topology.


\subsubsection{Sections which are locally Sobolev}
The pairing $C^\infty_K(\Sigma;S)\times C^\infty_c(\Sigma;S^*) \to \K$, $(u,\phi) \mapsto \int_\Sigma \<u,\phi\>\dS$, extends uniquely to a bicontinuous pairing $H^k_K(\Sigma;S)\times C^\infty_c(\Sigma;S^*) \to \K$.
Therefore we can consider Sobolev sections as distributional sections.
This yields a continuous embedding $H^k_K(\Sigma;S) \hookrightarrow \D'(\Sigma;S)$ for any compact subset $K\subset \Sigma$.
Thus we get a continuous embedding $H^k_c(\Sigma;S) \hookrightarrow \D'(\Sigma;S)$.
We put
\[
H^k\loc(\Sigma;S) := \{u\in\D'(\Sigma;S) \mid \chi u \in H^k_c(\Sigma;S) \mbox{ for all }\chi\in C^\infty_c(\Sigma,\R)\} \, .
\]
For every $\chi\in C^\infty_c(\Sigma,\R)$ we get the seminorm
\[
u \mapsto \|\chi u\|_{H^k(\supp(\chi))}
\]
on $H^k\loc(\Sigma;S)$.
We provide $H^k\loc(\Sigma;S)$ with the topology induced by these seminorms.
The same topology can be induced by a countable subfamily of seminorms.
Namely, choose a sequence of cutoff-functions $\chi_j$ such that the sets $\{\chi_j\equiv 1\}$ exhaust $M$.
The corresponding seminorms yield the same topology.
Hence $H^k\loc(\Sigma;S)$ is a Fr\'echet space.

Summarizing, we have the following chain of continuous inclusions
\[
C^\infty_c(\Sigma;S)
\subset
C^k_c(\Sigma;S)
\subset
H^k_c(\Sigma;S)
\subset
H^k\loc(\Sigma;S)
\subset
\D'(\Sigma;S)
\]
where the second space occurs only if $k\in\N_0$ while the Sobolev spaces are defined for all $k\in\R$.

\subsection{Finite energy sections}

Let $t:M\to\R$ be a Cauchy temporal function.
Fix $k\in\R$. 
The family $\{H^k\loc(\Sigma_s)\}_{s\in t(M)}$ is a bundle of Fr\'echet spaces over the interval $t(M)\subset \R$.
A global trivialization is given by parallel transport along the integral curves of the gradient vector field $\na t$.
We denote the space of $\ell$-times continuously differentiable sections of this bundle by $C^\ell(t(M),H^k\loc(\Sigma_\bullet))$.

\subsubsection{Embedding into $\D'(M;S)$}
Elements $u$ of $C^\ell(t(M),H^k\loc(\Sigma_\bullet))$ will be considered as distributional sections of $S$.
Namely, for any test section $\varphi\in C^\infty_c(M;S^*)$, the evaluation of $u$ on $\varphi$ is given by
\begin{equation}
u[\varphi] = \int_{t(M)} u(s)[(\beta^{1/2}\varphi)|_{\Sigma_s}] \, ds
\label{eq:SecDist}
\end{equation}
where $\beta:M\to\R$ is the function from \eqref{eq:gbeta}.
If $k\ge0$, then $u$ is locally integrable and we can rewrite \eqref{eq:SecDist} as
\[
u[\varphi] 
=
\int_{t(M)} \left(\int_{\Sigma_s} (\beta^{1/2}\varphi)|_{\Sigma_s}(x)\big(u(s)(x)\big) \, dA(x)\right) ds
=
\int_M \varphi u\, dV
\]
where $dA$ is the volume element of $\Sigma_s$ and $dV$ the one of $M$.
Here we observe $dV=\beta^{1/2}dA\,ds$ because of \eqref{eq:gbeta}.
This explains the factor $\beta^{1/2}$ in \eqref{eq:SecDist}.

\subsubsection{The topology of $C^\ell_{sc}(t(M),H^k(\Sigma_\bullet))$}
For any spatially compact subset $K\subset M$ we put
\[
C^\ell_K(t(M),H^k(\Sigma_\bullet))
 :=
\{u\in C^\ell(t(M),H^k\loc(\Sigma_\bullet))\mid \supp(u)\subset K\} .
\]
For any compact subinterval $I\subset t(M)$ we get the seminorm
\[
\|u\|_{I,K,\ell,k}
:=
\max_{i=0,\ldots,\ell} \,\,\max_{s\in I} \|\na^i_tu\|_{H^k(\Sigma_s)}
\]
on $C^\ell_K(t(M),H^k(\Sigma_\bullet))$.
Fixing $K$, $\ell$, and $k$, we let $I$ vary over all compact subintervals of $t(M)$ and turn $C^\ell_K(t(M),H^k(\Sigma_\bullet))$ into a Fr\'echet space.

Now we let $K$ vary over all spatially compact subsets of $M$ and we provide
\[
C^\ell_{sc}(t(M),H^k(\Sigma_\bullet))
:=
\bigcup_{K\subset M \atop {\mathrm{spatially} \atop \mathrm{compact}}} C^\ell_K(t(M),H^k(\Sigma_\bullet))
\]
with the strict inductive limit topology of locally convex spaces.
Hence the inclusion maps 
\[
C^\ell_K(t(M),H^k(\Sigma_\bullet)) \hookrightarrow C^\ell_{sc}(t(M),H^k(\Sigma_\bullet))
\] 
are continuous and for any locally convex topological vector space $X$, a linear map 
\[
\Lambda: C^\ell_{sc}(t(M),H^k(\Sigma_\bullet)) \to X
\] 
is continuous if and only if the restrictions of $\Lambda$ to all $C^\ell_K(t(M),H^k(\Sigma_\bullet))$ are continuous.

\begin{dfn}
For any $k\in\R$ we call
\[
\FE^k(M,t;S) 
:=
C^0_{sc}(t(M),H^k(\Sigma_\bullet)) \cap C^1_{sc}(t(M),H^{k-1}(\Sigma_\bullet))
\]
the space of \emph{finite $k$-energy sections}.
\end{dfn}

\subsubsection{The topology of $\FE^k(M,t;S)$}
Note that this space depends on the choice of Cauchy temporal function $t$.
A base of the topology of $\FE^k(M,t;S)$ is given by the sets of the form $U_0\cap U_1$ where $U_0$ is an open subset of $C^0_{sc}(t(M),H^k(\Sigma_\bullet))$ and $U_1$ is open in $C^1_{sc}(t(M),H^{k-1}(\Sigma_\bullet))$.
A map from a topological space 
\[
f: X \to \FE^k(M,t;S)
\] 
is continuous if and only if $f$ is continuous as a map $X \to C^0_{sc}(t(M),H^k(\Sigma_\bullet))$ and as a map $X \to C^1_{sc}(t(M),H^{k-1}(\Sigma_\bullet))$.
If we choose $f$ as the identity map on $\FE^k(M,t;S)$ we see that the inclusion maps $\FE^k(M,t;S)\hookrightarrow C^0_{sc}(t(M),H^k(\Sigma_\bullet))$ and $\FE^k(M,t;S)\hookrightarrow C^1_{sc}(t(M),H^{k-1}(\Sigma_\bullet))$ are continuous.

\subsubsection{The space $L^2\locsc(t(M),H^k(\Sigma_\bullet))$}
We will also need the space $L^2\locsc(t(M),H^k(\Sigma_\bullet))$ of $L^2\loc$-sections with spatially compact support.
By this we mean the following:
For any spatially compact subset $K\subset M$ the elements of $L^2\locK(t(M),H^k(\Sigma_\bullet))$ are those sections $u$ of the bundle $\{H^k\loc(\Sigma_s)\}_{s\in t(M)}$ for which we have
\begin{itemize}
\item 
$\supp(u(s)) \subset K\cap \Sigma_s$ for almost all $s\in t(M)$;
\item
the function $s\mapsto u(s)[\phi|_{\Sigma_s}]$ is measurable for any test section $\phi\in C^\infty_c(M;S^*)$;
\item
the function $s\mapsto \|u(s)\|_{H^k(\Sigma_s)}$ is $L^2\loc$.
\end{itemize}
This implies that the functions $s\mapsto u(s)[\phi|_{\Sigma_s}]$ are square integrable because they are compactly supported in $s$ and we have 
\[
|u(s)[\phi|_{\Sigma_s}]| 
\le 
C_1\cdot \|u(s)\|_{H^k(\Sigma_s)}\cdot \|\phi|_{\Sigma_s}\|_{H^{-k}(\Sigma_s)}
\le 
C_2(\phi)\cdot \|u(s)\|_{H^k(\Sigma_s)}.
\]
Again, we have an embedding $L^2\locK(t(M),H^k(\Sigma_\bullet)) \hookrightarrow \D'(M;S)$ via \eqref{eq:SecDist}.
The space $L^2\locK(t(M),H^k(\Sigma_\bullet))$ is topologized by the seminorms 
\[
\|u\|^2_{I,K,k} := \int_I \|u(s)\|^2_{H^k(\Sigma_s)}\, ds
\]
where $I$ runs through all compact subintervals of $t(M)$.
This turns $L^2\locK(t(M),H^k(\Sigma_\bullet))$ into a Fr\'echet space.
The space 
\[
L^2\locsc(t(M),H^k(\Sigma_\bullet))
:=
\bigcup_{K\subset M \atop {\mathrm{spatially} \atop \mathrm{compact}}} L^2\locK(t(M),H^k(\Sigma_\bullet))
\]
carries the strict inductive limit topology.

\begin{lem}\label{lem:Dicht}
Let $k\in\R$.
The space of smooth sections $C^\infty_{sc}(M;S)$ is dense in $L^2\locsc(t(M),H^k(\Sigma_\bullet))$.

If $K,K'\subset M$ are spatially compact and the interior of $K'$ contains $K$, then any $u\in L^2\locsc(t(M),H^k(\Sigma_\bullet))$ with $\supp(u)\subset K$ can be approximated by smooth sections with support contained in $K'$.
\end{lem}

\begin{proof}
W.l.o.g.\ we assume that $M$ is spatially compact, otherwise use the doubling trick described in Subsection~\ref{subsec:Sobolev}.
Let $u\in L^2\locsc(t(M),H^k(\Sigma_\bullet))$ with $\supp(u)\subset K$.
Let $\chi\in C^\infty_c(M,\R)$ be a cutoff function with $\chi\equiv1$ on $K$ and $\supp(\chi)\subset \mathrm{interior}(K')$.

We first mollify in spatial directions by setting
\[
u_\eps(s) := (\chi|_{\Sigma_s}) \cdot e^{-\eps\Db^2}(u(s)).
\]
Then $u_\eps(s)\in C^\infty(\Sigma_s)$ and $u_\eps(s) \to u(s)$ in $H^k(\Sigma_s)$ as $\eps\searrow0$, locally uniformly in $s$.

Now we fix a compact subinterval $I\subset t(M)$.
We choose a nonnegative function $\rho\in C^\infty_c(\R,\R)$ such that $\rho(s)=0$ for $|s|\ge1$ and $\int_\R \rho(s)ds=1$.
We put
\[
u_{\delta,\eps}(s) 
:= 
\frac{1}{\delta}\int_I \rho\bigg(\frac{s-\sigma}{\delta}\bigg) \Pi^\sigma_s (u_\eps(\sigma))\, d\sigma .
\]
Here $\Pi^\sigma_s (u_\eps(\sigma))$ is the parallel translate of $u_\eps(\sigma)$ to $\Sigma_s$ along the gradient lines of the Cauchy temporal function $t$.
Then $u_{\delta,\eps}\in C^\infty(M;S)$ and $u_{\delta,\eps} \to u_\eps$ in $L^2(I,H^k(\Sigma_\bullet))$ as $\delta\searrow0$.
For any $\eps$ and sufficiently small $\delta$ we have $\supp(u_{\delta,\eps})\subset K'$.
\end{proof}

\subsubsection{Dependence on the time function}\label{sss:DepTime}
The space $\FE^k(M,t;S)$ depends in general on the choice of Cauchy temporal function.
For example, let $M$ be the $(1+1)$-dimensional Minkowski space with standard coordinates $x_0,x_1$.
We choose $t=x_0 : M \to \R$.
Now let $g\in C^\infty_c(\R,\R)$ be such that $g\equiv 1$ on $[-2,2]$.
Let $f\in C^1(\R,\R)$ such that $f$ is nowhere $C^2$.
Then the function $u:M\to\R$, $u(t,x)=f(t)g(x)$, is in $C^1(\R,H^k(\R))$ for all $k\in\R$.

Now we change the time function by composing with a nontrivial Lorentz boost $L=\begin{pmatrix} \cosh(\theta) & \sinh(\theta) \\ \sinh(\theta) & \cosh(\theta) \end{pmatrix}$ where $\theta\neq 0$.
The transformed function $\tilde u(t,x) = (u\circ L)(t,x) = f(\cosh(\theta)t + \sinh(\theta)x)g(\sinh(\theta)t + \cosh(\theta)x)$ is no longer in $C^1(\R,H^k(\R))$ for $k>2+\tfrac12$.
Namely, if it were, then for constant $t$ near $0$ the function $x\mapsto \tilde u(t,x) = f(\cosh(\theta)t + \sinh(\theta)x)$, $x\in [-1,1]$, would be $C^2$ by the Sobolev embedding theorem, contradicting the assumption on $f$.

A similar argument shows that the space $L^2\locsc(t(M),H^k(\Sigma_\bullet))$ depends on the Cauchy temporal function $t$ if $k>0$.
For $k=0$ the situation is different.
Namely
\[
L^2\locsc(t(M),H^0(\Sigma_\bullet))
=
L^2\locsc(t(M),L^2(\Sigma_\bullet))
=
L^2\locsc(M)
\]
because
\[
\| \, \|u(s)\|^2_{L^2(\Sigma_s)} \, \|^2_{L^2(I)}
= 
\int_I \int_{\Sigma_s} |u|^2 dA\, ds
\sim
\int_I \int_{\Sigma_s} \beta^{1/2} |u|^2 dA\, ds
=
\int_{t^{-1}(I)} |u|^2 dV \, .
\]
Here $I$ is any compact subinterval of $t(M)$ and ``$\sim$'' means that both sides can be estimated against each other up to multiplicative constants depending on $I$ and the support of $u$ but not on $u$ itself.
Thus $L^2\locsc(t(M),H^0(\Sigma_\bullet)) = L^2\locsc(M)$ does not depend on the choice of $t$.

\subsubsection{The spaces $\FE^k(M,P,t)$ and $\FE^k(M,\ker(P),t)$}
We will have to apply differential operators to sections in $\FE^k(M,t)$ and will have to make sure that they map them to sections of the right regularity.
This requires the introduction of suitable subspaces of $\FE^k(M,t)$.

\begin{dfn}
Let $P$ be a linear differential operator (with smooth coefficients) acting on sections of $S$.
We put
\begin{align*}
\FE^k(M,P,t) &:= \FE^k(M,t;S) \cap P^{-1}(L^2\locsc(t(M),H^{k-1}(\Sigma_\bullet))), \\
\FE^k(M,\ker(P),t) &:= \FE^k(M,t;S) \cap \ker(P).
\end{align*}
Elements of the latter space will be called \emph{finite $k$-energy solutions} of $P$.
Here $P^{-1}(L^2\loc(t(M),H^{k-1}(\Sigma_\bullet)))$ and $\ker(P)$ denote the spaces of all $u\in\D'(M;S)$ such that $Pu\in L^2\loc(t(M),H^{k-1}(\Sigma_\bullet))$ and $Pu=0$ in the distributional sense, respectively.
\end{dfn}
A base of the topology of $\FE^k(M,P,t)$ is given by the sets $U\cap P^{-1}(V)$ where $U$ is an open subset of $\FE^k(M,t;S)$ and $V$ is an open subset of $L^2\loc(t(M),H^{k-1}(\Sigma_\bullet))$.
A map from a topological space 
\[
f: X \to \FE^k(M,P,t)
\] 
is continuous if and only if $f$ is continuous as a map $X \to \FE^k(M,t;S)$ and $P\circ f:X\to C^0_{sc}(t(M),H^{k-1}(\Sigma_\bullet))$ is continuous.
In particular, $P:\FE^k(M,P,t)\to L^2\loc(t(M),H^{k-1}(\Sigma_\bullet))$ is continuous.

Although $\FE^k(M,P,t)$ does not carry the relative topology induced by $\FE^k(M,t;S)$ (the topology of $\FE^k(M,P,t)$ is finer), both spaces, $\FE^k(M,t;S)$ and $\FE^k(M,P,t)$, induce the same relative topology on $\FE^k(M,\ker(P),t)$.

\subsubsection{Dependence on the time function}
In general, the spaces $\FE^k(M,P,t)$ do depend on the choice of Cauchy temporal function $t$.
To see this, let $M$ again be the $(1+1)$-dimensional Minkowski space with standard coordinates $x_0,x_1$ and $t=x_0 : M \to \R$.
Let $P=-\frac{\partial^2}{\partial x_0^2} + \frac{\partial^2}{\partial x_1^2}$.
Let $g\in C^\infty_c(\R,\R)$ be such that $g\equiv 1$ on $[-2,2]$.
Let $f\in C^2(\R,\R)$ such that $f$ is nowhere $C^3$.
Then the function $u:M\to\R$, $u(t,x)=f(t)g(x)$, is in $C^2(\R,H^k(\R))$ for all $k\in\R$ and $Pu \in C^0(\R,H^k(\R))$ for all $k\in\R$.
Thus $u\in \FE^k(M,P,t)$ for all $k\in\R$.

Again, we change the time function by composing with a nontrivial Lorentz boost.
The transformed function $\tilde u(t,x) = (u\circ L)(t,x) = f(\cosh(\theta)t + \sinh(\theta)x)g(\sinh(\theta)t + \cosh(\theta)x)$ is no longer in $C^0(\R,H^k(\R))$ for $k>3+\tfrac12$.
Namely, if it were, then for constant $t$ near $0$ the function $x\mapsto \tilde u(t,x) = f(\cosh(\theta)t + \sinh(\theta)x)$, $x\in [-1,1]$, would be $C^3$ by the Sobolev embedding theorem, contradicting the assumption on $f$.
Thus $\tilde u\notin \FE^k(M,P,t)$ for $k>3+\tfrac12$.

On the other hand, if $P$ is a wave operator, \cref{cor:FEPIndep} will show that the space $\FE^1(M,P,t)$ is independent of the choice of $t$.
Moreover, by \cref{cor:FEkerPIndep}, all spaces $\FE^k(M,\ker(P),t)$ are independent of $t$.

\subsection{Wave operators}

Let $M$ be equipped with a Lorentzian metric $g$ and let $S\to M$ be a real or complex vector bundle.
A linear differential operator of second order $P:C^\infty(M;S)\to C^\infty(M;S)$ is called a \emph{wave operator} if its principal symbol is given by $g$.
In other words, $P$ locally takes the form
\[
P = -\sum_{ij} g^{ij} \frac{\partial^2}{\partial x^i\partial x^j} + \mbox{ lower order terms} \,.
\]
Here $(g^{ij})$ is the matrix inverse to $(g_{ij})$ where $g_{ij}=g(\partial/\partial x^i,\partial/\partial x^j)$.

If $P$ is a wave operator, then so are its formally dual operator $P^\dagger$ and its formally adjoint operator $P^*$ (if $S$ carries a metric).

\begin{ex}
Let $S$ be the trivial line bundle so that sections of $S$ are just functions.
The \emph{d'Alembert operator} $P=\Box=-\mathrm{div}\circ\mathrm{grad}$ is a formally selfadjoint wave operator, see e.g.\ \cite[p.~26]{BGP07}.
Similarly, the \emph{Klein-Gordon operator} $P=\Box + m^2$ and the \emph{covariant Klein-Gordon operator}  $P= \Box + m^2 + \xi\cdot\mathrm{scal}$ are formally selfadjoint wave operators.
Here $m$ and $\xi$ are constants and $\mathrm{scal}$ denotes the scalar curvature of $M$.
\end{ex}

\begin{ex}\label{ex:connectiondAlembert}
More generally, let $S$ be any vector bundle and let $\nabla$ be a connection on $S$.
This connection and the Levi-Civita connection on $T^*M$ induce a connection on $T^*M\otimes S$, again denoted $\nabla$.
We define the \emph{connection-d'Alembert} operator $\Box^\nabla$ to be the composition of the following three maps
$$
 C^\infty(M;S) \xrightarrow{\nabla}
 C^\infty(M,T^*M\otimes S) \xrightarrow{\nabla}
 C^\infty(M,T^*M\otimes T^*M\otimes S)
\xrightarrow{-\tr\otimes\id_S} 
 C^\infty(M;S)
$$
where $\tr:T^*M\otimes T^*M \to \R$ denotes the metric trace,
$\tr(\xi\otimes\eta)=\<\xi,\eta\>$. 
This connection-d'Alembert operator $\Box^\nabla$ is a wave operator.
\end{ex}

\begin{ex}
Let $S=\Lambda^kT^*M$ be the bundle of $k$-forms.
Exterior differentiation $d: C^\infty(M,\Lambda^kT^*M) \to C^\infty(M,\Lambda^{k+1}T^*M)$ increases the degree by one while
the codifferential $\delta: C^\infty(M,\Lambda^{k}T^*M) \to C^\infty(M,\Lambda^{k-1}T^*M)$ decreases the degree by one.
While $d$ is independent of the metric, the codifferential $\delta$ does depend on the Lorentzian metric.
The \emph{Hodge-d'Alembert operator} $P=-d\delta - \delta d$ is a wave operator.
It is formally selfadjoint with respect to the indefinite metric on $S$ induced by the Lorentzian metric on $TM$.
\end{ex}


\section{The energy estimate}

Let $(M,g)$ be a globally hyperbolic Lorentzian manifold of dimension $n+1$. 
Let $S\to M$ be a real or complex vector bundle and let $P$ be a linear wave operator with smooth coefficients acting on sections of $S$.
We equip $S$ with a metric $\<\cdot,\cdot\>$, the associated norm $|\cdot|$ and a compatible connection $\na$.

\begin{dfn}
Let $\Sigma\subset M$ be a smooth spacelike Cauchy hypersurface, let $\nu$ be the future-directed timelike unit-normal field along $\Sigma$ (in particular, $g(\nu,\nu)=-1$), and let $u$ be a sufficiently differentiable section of $S$ defined on a neighborhood of $\Sigma$.
For $k\in\R$ we call
\[
E_k(u,\Sigma) := \|u|_\Sigma\|^2_{H^{k}(\Sigma)} + \|\na_\nu u\|^2_{H^{k-1}(\Sigma)}
\]
the \emph{$k$-energy} of $u$ along $\Sigma$.
\end{dfn}

Given a Cauchy temporal function $t$, we briefly write
\[
E_k(u,s) 
:= 
E_k(u,\Sigma_s)
=
\|u|_{\Sigma_{s}}\|^2_{H^{k}(\Sigma_s)} + \|\beta^{-1/2}\na_t u\|^2_{H^{k-1}(\Sigma_s)}.
\]

\begin{thm}[Energy estimate]\label{thm:enest}
Let $[T_0,T_1]\subset t(M)$, let $K\subset M$ be compact, and let $k\in\R$.
Then there exists a constant $C>0$ such that 
\begin{equation}
E_k(u,t_1)
\leq
E_k(u,t_0) \cdot e^{C(t_1-t_0)} + \int_{t_0}^{t_1} e^{C(t_1-s)} \|Pu\|^2_{H^{k-1}(\Sigma_s)}\, ds 
\label{eq:EnEst}
\end{equation}
holds for all $t_0,t_1\in[T_0,T_1]$ with $t_0<t_1$ and for all $u\in \FE^{k+1}(M,t;S)$ with $\supp(u)\subset J(K)$ and $Pu\in C^0_{sc}(t(M),H^{k-1}(\Sigma_\bullet))$.
\end{thm}

\begin{rem}
The regularity assumptions on $u$ seem somewhat unnatural.
Considering the terms occuring in \eqref{eq:EnEst}, one would expect it to hold for $u\in \FE^{k}(M,t;S)$ instead of $\FE^{k+1}(M,t;S)$ and $Pu\in L^2\locsc(t(M),H^{k-1}(\Sigma_\bullet))$ instead of $C^0_{sc}(t(M),H^{k-1}(\Sigma_\bullet))$.
In other words, we expect it to hold for $u\in \FE^k(M,P,t)$.

Indeed, this is true as we will see in \cref{cor:EnEstOptimal}.
As far as the regularity assumption on $u$ is concerned, \tref{thm:enest} is a preliminary version of the energy estimate.
\end{rem}

\begin{proof}[Proof of \tref{thm:enest}]
With the doubling procedure described in Subsection~\ref{subsec:Sobolev}, we can reduce to the case that $M$ is spatially compact which we now assume.
We write the wave operator as follows:
\begin{align*}
P
&=
\beta^{-1} \na_t\na_t - \sum_{j=1}^n \na_{e_j}\na_{e_j} + \mbox{ lower order terms} \\
&=
\beta^{-1} \na_t\na_t + \Db^2 + \Zb\na_t + \Fb 
\end{align*}
where $\Fb$ and $\Zb$ are differential operators differentiating in $\Sigma_t$-direction only, of order at most $1$ and $0$, respectively.

Since $u\in\FE^{k+1}(M,t;S)$ we have $u\in C^1_{sc}(t(M),H^{k}(\Sigma_\bullet))$ and hence $s\mapsto \|u\|^2_{H^{k}(\Sigma_s)}$ is differentiable.
From $Pu\in C^0_{sc}(t(M),H^{k-1}(\Sigma_\bullet))$ and 
\begin{equation}
\na_t\na_t u = \beta\left\{Pu - \Db^2u - \Zb\na_tu - \Fb u \right\}
\label{eq:Solu}
\end{equation}
we get $\na_t\na_t u \in C^0_{sc}(t(M),H^{k-1}(\Sigma_\bullet))$.
Together with $\na_t u\in C^0_{sc}(t(M),H^{k}(\Sigma_\bullet))$ this gives $\na_t u\in C^1_{sc}(t(M),H^{k-1}(\Sigma_\bullet))$.
Hence $s\mapsto \|\beta^{-1/2}\na_tu\|^2_{H^{k-1}(\Sigma_s)}$ is differentiable.
We differentiate
\begin{align}
\ddt \|u\|^2_{H^{k}(\Sigma_t)}
&=
2\Re (\na_t\Db^ku,\Db^ku)_{L^2(\Sigma_t)} \notag\\
&=
2\Re (\Db^k\na_tu + [\na_t,\Db^k]u,\Db^ku)_{L^2(\Sigma_t)} \notag\\
&\le
2\Re (\Db^k\na_tu,\Db^ku)_{L^2(\Sigma_t)} + C_1 \|u\|^2_{H^k(\Sigma_t)} \, .
\label{eq:EE1}
\end{align}
Here we used that the commutator $[\na_t,\Db^k]$ is a smooth family of pseudo-differential operators along the $\Sigma_s$ of order $\le k$.
Next we differentiate
\begin{align}
\ddt &\|\beta^{-1/2}\na_t u\|^2_{H^{k-1}(\Sigma_s)} \\
&=
2\Re (\na_t\Db^{k-1}(\beta^{-1/2}\na_tu),\Db^{k-1}(\beta^{-1/2}\na_tu))_{L^2(\Sigma_t)} \notag\\
&=
2\Re (\Db^{k-1}\na_t(\beta^{-1/2}\na_tu),\Db^{k-1}(\beta^{-1/2}\na_tu))_{L^2(\Sigma_t)} \notag\\
&\quad
+2\Re([\na_t,\Db^{k-1}](\beta^{-1/2}\na_tu),\Db^{k-1}(\beta^{-1/2}\na_tu))_{L^2(\Sigma_t)} \notag\\
&\le
2\Re (\Db^{k-1}\na_t(\beta^{-1/2}\na_tu),\Db^{k-1}(\beta^{-1/2}\na_tu))_{L^2(\Sigma_t)} 
+ C_2\|\beta^{-1/2}\na_tu\|^2_{H^{k-1}(\Sigma_t)} \notag\\
&=
2\Re (\Db^{k-1}(\partial_t\beta^{-1/2}+\beta^{-1/2}\na_t)\na_tu,\Db^{k-1}(\beta^{-1/2}\na_tu))_{L^2(\Sigma_t)}  \notag\\
&\quad + C_2\|\beta^{-1/2}\na_tu\|^2_{H^{k-1}(\Sigma_t)}\notag\\
&\le
2\Re (\Db^{k-1}(\beta^{-1/2}\na_t\na_tu),\Db^{k-1}(\beta^{-1/2}\na_tu))_{L^2(\Sigma_t)} 
+ C_3\|\beta^{-1/2}\na_tu\|^2_{H^{k-1}(\Sigma_t)} \, .
\label{eq:EE2}
\end{align}
Using \eqref{eq:Solu} we get for the first summand in \eqref{eq:EE2}:
\begin{align}
\Re &(\Db^{k-1}(\beta^{-1/2}\na_t\na_tu),\Db^{k-1}(\beta^{-1/2}\na_tu))_{L^2(\Sigma_t)} \notag\\
&=
\Re (\Db^{k-1}(\beta^{1/2}\{Pu - \Db^2u - \Zb\na_tu - \Fb u\}),\Db^{k-1}(\beta^{-1/2}\na_tu))_{L^2(\Sigma_t)} \notag\\
&\le
-\Re(\Db^{k-1}(\beta^{1/2}\Db^2u),\Db^{k-1}(\beta^{-1/2}\na_tu))_{L^2(\Sigma_t)}  \notag\\
&\quad
+ C_4\cdot[\|Pu\|_{H^{k-1}(\Sigma_t)}+\|\beta^{-1/2}\na_tu\|_{H^{k-1}(\Sigma_t)} + \|u\|_{H^{k}(\Sigma_t)}]\cdot\|\beta^{-1/2}\na_tu\|_{H^{k-1}(\Sigma_t)}
\notag\\
&\le
-\Re(\Db(\beta^{1/2}\Db^{k}u),\Db^{k-1}(\beta^{-1/2}\na_tu))_{L^2(\Sigma_t)}  \notag\\
&\quad
+ C_5\cdot[\|Pu\|_{H^{k-1}(\Sigma_t)}+\|\beta^{-1/2}\na_tu\|_{H^{k-1}(\Sigma_t)} + \|u\|_{H^{k}(\Sigma_t)}]\cdot\|\beta^{-1/2}\na_tu\|_{H^{k-1}(\Sigma_t)}
\notag\\
&=
-\Re(\Db^{k}u,\beta^{1/2}\Db^{k}(\beta^{-1/2}\na_tu))_{L^2(\Sigma_t)}  \notag\\
&\quad
+ C_5\cdot[\|Pu\|_{H^{k-1}(\Sigma_t)}+\|\beta^{-1/2}\na_tu\|_{H^{k-1}(\Sigma_t)} + \|u\|_{H^{k}(\Sigma_t)}]\cdot\|\beta^{-1/2}\na_tu\|_{H^{k-1}(\Sigma_t)}
\notag\\
&\le
-\Re(\Db^{k}u,\Db^{k}\na_tu)_{L^2(\Sigma_t)}  \notag\\
&\quad
+ C_6\cdot[\|Pu\|_{H^{k-1}(\Sigma_t)}+\|\beta^{-1/2}\na_tu\|_{H^{k-1}(\Sigma_t)} + \|u\|_{H^{k}(\Sigma_t)}]\cdot\|\beta^{-1/2}\na_tu\|_{H^{k-1}(\Sigma_t)} \, . \notag
\end{align}
Inserting this into \eqref{eq:EE2} yields
\begin{equation}
\ddt \|\beta^{-1/2}\na_t u\|^2_{H^{k-1}(\Sigma_s)}
\le
-2\Re(\Db^{k}u,\Db^{k}\na_tu)_{L^2(\Sigma_t)} + C_7 \cdot E_k(u,t) + \|Pu\|^2_{H^{k-1}(\Sigma_t)}.
\label{eq:EE3}
\end{equation}
We add \eqref{eq:EE1} and \eqref{eq:EE3} and observe that the term ``with too many derivatives'', $2\Re(\Db^{k}u,\Db^{k}\na_tu)_{L^2(\Sigma_t)}$, cancels.
We get
\begin{equation}
\ddt E_k(u,t) \le C_8\cdot E_k(u,t) + \|Pu\|^2_{H^{k-1}(\Sigma_t)}.
\end{equation}
Gr\"onwall's lemma implies
\begin{align*}
E_k(u,t_1)
&\leq
E_k(u,t_0) \cdot e^{C_8(t_1-t_0)}  
+ \int_{t_0}^{t_1} e^{C_8(t_1-s)} \|Pu\|^2_{H^{k-1}(\Sigma_s)}\, ds . \qedhere
\end{align*}
\end{proof}

For $[t_0,t_1]\subset t(M)$ we write $\Sigma[t_0,t_1] := t^{-1}([t_0,t_1])$.

\begin{cor}\label{cor:sliceest}
Let $[T_0,T_1]\subset t(M)$, let $\tau\in t(M)$, and let $K\subset \Sigma_\tau$ be compact.
Let $k\in\R$.
Then there exists a constant $C>0$ such that 
\[
E_k(u,\Sigma_s)
\leq
C\cdot \left(E_k(u,\Sigma_\tau) + \|Pu\|^2_{[T_0,T_1],J(K),k-1} \right)
\]
holds for all $s\in[T_0,T_1]$ and for all $u\in \FE^{k+1}(M,P,t)$ with $Pu\in C^0_{sc}(t(M),H^{k-1}(\Sigma_\bullet))$ and $\supp(u)\subset J(K)$.
\end{cor}

\begin{proof}
W.l.o.g.\ we can assume that $\tau\in[T_0,T_1]$, otherwise we simply increase the interval $[T_0,T_1]$.
For $s\in [\tau,T_1]$ the assertion follows directly from Theorem~\ref{thm:enest}.
For $s\in [T_0,\tau]$, the assertion follows from Theorem~\ref{thm:enest} after reversing the time orientation.
\end{proof}

\begin{cor}[Uniqueness for the Cauchy problem]\label{cor:CauchyUnique}
A section $u\in\FE^k(M,P,t)$ is uniquely determined by $Pu$ and the restrictions of $u$ and of $\na_\nu u$ to any $\Sigma_\tau$.
\end{cor}

\begin{proof}
If $u$ and $\na_\nu u$ vanish along $\Sigma_\tau$ and $Pu=0$, then by \cref{cor:sliceest} $E_{k-1}(u,\Sigma_s)=0$ for all $s\in t(M)$, hence $u=0$.
\end{proof}

\begin{cor}\label{cor:Hkest}
Let $[T_0,T_1]\subset t(M)$ and let $K\subset M$ be compact.
Let $k\in\N$.
Then there exists a constant $C>0$ such that 
\[
\|u\|^2_{H^k(\Sigma[t_0,t_1])}
\leq
C\cdot\left(\|u|_{\Sigma_{t_0}}\|^2_{H^k(\Sigma_{t_0})}
+ \|\na_tu\|^2_{H^{k-1}(\Sigma_{t_0})}
+ \|Pu\|^2_{H^{k-1}(\Sigma[t_0,t_1])}\right)
\]
holds for all $t_0,t_1\in[T_0,T_1]$ with $t_0<t_1$ and for all $u\in C^\infty_{J(K)}(M)$.
\end{cor}

\begin{proof}
Recall that $P$ is of the form
\[
P = \beta^{-1}\na_t\na_t + Q
\]
where $Q$ is a second-order operator containing $t$-derivatives up to order $1$ only.
Therefore
\begin{align*}
\|(\na_t)^\ell u\|_{H^{k-\ell}(\Sigma[t_0,t_1])}
&=
\|(\na_t)^{\ell-2}(\beta(P-Q)) u\|_{H^{k-\ell}(\Sigma[t_0,t_1])} \\
&\le
C_1\cdot\left(\|Pu\|_{H^{k-2}(\Sigma[t_0,t_1])} + \|(\na_t)^{\ell-2}(\beta Q u)\|_{H^{k-\ell}(\Sigma[t_0,t_1])}\right)
\end{align*}
where $(\na_t)^{\ell-2}(\beta Q u)$ contains at most $\ell-1$ $t$-derivatives of $u$.
The $k$-energy controls $t$-derivatives up to order $1$.
We get inductively
\begin{align}
\|u\|^2_{H^k(\Sigma[t_0,t_1])}
&\le
C_2\cdot\left( \int_{t_0}^{t_1} E_k(u,s)\,ds + \sum_{\ell=2}^k\|(\na_t)^\ell u\|^2_{H^{k-\ell}(\Sigma[t_0,t_1])} \right) \notag\\
&\le
C_3\cdot\left( \int_{t_0}^{t_1} E_k(u,s)\,ds + \sum_{\ell=2}^{k-1}\|(\na_t)^\ell u\|^2_{H^{k-\ell}(\Sigma[t_0,t_1])}  + \|Pu\|_{H^{k-2}(\Sigma[t_0,t_1])}^2 \right) \notag\\
&\le
\quad\quad\cdots \notag\\
&\le
C_4 \cdot\left( \int_{t_0}^{t_1} E_k(u,s)\,ds  + \|Pu\|_{H^{k-2}(\Sigma[t_0,t_1])}^2 \right) .
\label{eq:Hk1}
\end{align}
\tref{thm:enest} yields
\begin{align}
\int_{t_0}^{t_1} E_k(u,s)\,ds 
&\le
\int_{t_0}^{t_1} \left(E_k(u,t_0) \cdot e^{C_5(s-t_0)}  + \int_{t_0}^{s} e^{C_5(s-t)} \|Pu\|_{H^{k-1}(\Sigma_t)}^2 \, dt\right)ds \notag\\
&\le
(t_1-t_0) \cdot e^{C_5(t_1-t_0)} \cdot \left(E_k(u,t_0) + \int_{t_0}^{t_1} \|Pu\|_{H^{k-1}(\Sigma_t)}^2\, dt\right) \notag\\
&\le
C_6\cdot\left(\|u|_{\Sigma_{t_0}}\|^2_{H^k(\Sigma_{t_0})}
+ \|\na_tu\|^2_{H^{k-1}(\Sigma_{t_0})}
+ \|Pu\|^2_{H^{k-1}(\Sigma[t_0,t_1])}\right) .
\label{eq:Hk2}
\end{align}
Inserting \eqref{eq:Hk2} into \eqref{eq:Hk1} proves the corollary.
\end{proof}

\section{The Cauchy problem}

The following theorem states that the Cauchy problem for $P$ is well posed in spaces of finite $k$-energy sections.

\begin{thm}\label{thm:Cauchy1}
Fix $\tau\in t(M)$ and $k\in\R$.
The map $u\mapsto (u|_{\Sigma_\tau},\na_\nu u|_{\Sigma_\tau},Pu)$ yields an isomorphism
\begin{equation}
\FE^k(M,P,t)
\to
H^k_c(\Sigma_\tau) \oplus H^{k-1}_c(\Sigma_\tau) \oplus L^2\locsc(t(M),H^{k-1}(\Sigma_\bullet))
\label{eq:Cauchy1}
\end{equation}
of topological vector spaces.
\end{thm}

\begin{proof}
For any spatially compact subset $K\subset M$ we have the continuous linear maps
\[
C^0_K(t(M),H^k(\Sigma_\bullet)) 
\xrightarrow{u\mapsto u|_{\Sigma_\tau}}
H^k_{K\cap\Sigma_\tau}(\Sigma_\tau)
\hookrightarrow
H^k_c(\Sigma_\tau) .
\]
Hence
\[
\FE^k(M,t)
\hookrightarrow
C^0_{sc}(t(M),H^k(\Sigma_\bullet)) 
\rightarrow
H^k_c(\Sigma_\tau)
\]
is continuous and thus the restriction to $\FE^k(M,P,t)$ is continuous as well.

Similarly, 
\[
C^1_K(t(M),H^{k-1}(\Sigma_\bullet))
\xrightarrow{\na_\nu}
C^0_K(t(M),H^{k-1}(\Sigma_\bullet)) 
\xrightarrow{u\mapsto u|_{\Sigma_\tau}}
H^{k-1}_{K\cap\Sigma_\tau}(\Sigma_\tau)
\hookrightarrow
H^{k-1}_c(\Sigma_\tau) 
\]
is continuous which yields continuity of the linear maps
\[
\FE^k(M,P,t)
\hookrightarrow
C^1_{sc}(t(M),H^{k-1}(\Sigma_\bullet)) 
\rightarrow
H^{k-1}_c(\Sigma_\tau)  .
\]
Since $P:\FE^k(M,P,t)\to L^2\locsc(t(M),H^{k-1}(\Sigma_\bullet))$ is continuous as well, we have shown that the map in \eqref{eq:Cauchy1} is continuous.

Now we consider the inverse of the map in \eqref{eq:Cauchy1}.
Let $K\subset\Sigma_\tau$ be compact.
Given $(u_0,u_1,f) \in C^\infty_K(\Sigma_\tau) \oplus C^\infty_K(\Sigma_\tau) \oplus C^\infty_{J(K)}(M)$, there is a unique smooth solution $u$ of $Pu=f$ with $u|_{\Sigma_\tau}$ and $\na_\nu u|_{\Sigma_\tau}=u_1$, see \cite[Thm.~3.2.11]{BGP07}.
It satisfies $\supp(u) \subset J(K)$.

Let $I\subset t(M)$ be a compact subinterval.
By Corollary~\ref{cor:sliceest}, we have estimates
\[
\|u\|^2_{I,J(K),0,k} 
\le 
C\cdot \left( \|u_0\|^2_{H^k(\Sigma_\tau)} + \|u_1\|^2_{H^{k-1}(\Sigma_\tau)} + \|f\|^2_{I,J(K),k-1} \right)
\]
and
\[
\|u\|^2_{I,J(K),1,k-1} 
\le 
C\cdot \left( \|u_0\|^2_{H^k(\Sigma_\tau)} + \|u_1\|^2_{H^{k-1}(\Sigma_\tau)} + \|f\|^2_{I,J(K),k-1} \right)
\]
with $C$ depending on $K$ but independent of $u_0$, $u_1$, and $f$.
Thus the solution map $(u_0,u_1,f)\mapsto u$ extends uniquely to continuous linear maps
\begin{align*}
H^k_K(\Sigma_\tau) \oplus H^{k-1}_K(\Sigma_\tau) \oplus L^2\locJK(t(M),H^{k-1}(\Sigma_\bullet))
&\rightarrow
C^0_{J(K)}(t(M),H^k(\Sigma_\bullet)) \\
&\hookrightarrow
C^0_{sc}(t(M),H^k(\Sigma_\bullet))
\end{align*}
and
\begin{align*}
H^k_K(\Sigma_\tau) \oplus H^{k-1}_K(\Sigma_\tau) \oplus L^2\locJK(t(M),H^{k-1}(\Sigma_\bullet))
&\rightarrow
C^1_{J(K)}(t(M),H^{k-1}(\Sigma_\bullet)) \\
&\hookrightarrow
C^1_{sc}(t(M),H^{k-1}(\Sigma_\bullet))  \, .
\end{align*}
Thus we get continuous linear maps
\[
H^k_c(\Sigma_\tau) \oplus H^{k-1}_c(\Sigma_\tau) \oplus L^2\locsc(t(M),H^{k-1}(\Sigma_\bullet))
\rightarrow
C^0_{sc}(t(M),H^{k}(\Sigma_\bullet))  
\]
and
\[
H^k_c(\Sigma_\tau) \oplus H^{k-1}_c(\Sigma_\tau) \oplus L^2\locsc(t(M),H^{k-1}(\Sigma_\bullet))
\rightarrow
C^1_{sc}(t(M),H^{k-1}(\Sigma_\bullet))  
\]
and hence a continuous linear map
\[
\Solve : 
H^k_c(\Sigma_\tau) \oplus H^{k-1}_c(\Sigma_\tau) \oplus L^2\locsc(t(M),H^{k-1}(\Sigma_\bullet))
\rightarrow
\FE^k(M,t)  \,  .
\]
Since $P\circ\Solve(u_0,u_1,f) = f$, the composition $P\circ\Solve$ is continuous.
Thus the solution map is also continuous as a map
\[
\Solve : 
H^k_c(\Sigma_\tau) \oplus H^{k-1}_c(\Sigma_\tau) \oplus L^2\locsc(t(M),H^{k-1}(\Sigma_\bullet))
\rightarrow
\FE^k(M,P,t).
\]
The map \eqref{eq:Cauchy1} and $\Solve$ are inverse to each other.
Indeed, $\Solve$ followed by the map in \eqref{eq:Cauchy1} is the identity.
Conversely, if we start with a solution, apply \eqref{eq:Cauchy1} and solve again, we recover the original solution because of \cref{cor:CauchyUnique}.
\end{proof}

In particular, the Cauchy problem for the homogeneous equation $Pu=0$ is well posed:

\begin{cor}\label{cor:Cauchy2}
Fix $\tau\in t(M)$ and $k\in\R$.
The map $u\mapsto (u|_{\Sigma_\tau},\na_\nu u|_{\Sigma_\tau})$ yields an isomorphism
\begin{equation}
\FE^k(M,\ker(P),t)
\to
H^k_c(\Sigma_\tau) \oplus H^{k-1}_c(\Sigma_\tau)
\label{eq:Cauchy2}
\end{equation}
of topological vector spaces.\qed
\end{cor}

As a first consequence we see that smooth sections are dense in $\FE^k(M,P,t)$ and in $\FE^k(M,\ker(P),t)$.

\begin{cor}\label{cor:dense}
For any $k\in\R$, $C^\infty_{sc}(M)$ is dense in $\FE^k(M,P,t)$ and $C^\infty_{sc}(M) \cap \ker(P)$ is dense in $\FE^k(M,\ker(P),t)$.
\end{cor}

\begin{proof}
This follows from \tref{thm:Cauchy1} and \cref{cor:Cauchy2} because $C^\infty_c(\Sigma_\tau)$ is dense in $H^k_c(\Sigma_\tau)$ and in $H^{k-1}_c(\Sigma_\tau)$ and $C^\infty_{sc}(M)$ is dense in $L^2\locsc(t(M),H^{k-1}(\Sigma_\bullet))$ by \lref{lem:Dicht} and any solution with smooth Cauchy data is smooth by \cite[Thm.~3.2.11]{BGP07}.
\end{proof}

\begin{rem}
As noted in the proof of \tref{thm:Cauchy1}, it is well known that any smooth solution $u$ of the Cauchy problem $Pu=f$, $u|_\Sigma=u_0$, $\na_\nu u|_\Sigma=u_1$ satisfies $\supp(u)\subset J(K)$ if $\supp(u_0) \cup \supp(u_1) \subset K$ and $\supp(f)\subset J(K)$ where $K$ is a compact subset of $\Sigma$.
Since smooth sections are dense in $\FE^k(M,P,t)$, the same is true for all solutions of finite $k$-energy by continuity.
This fact is known as \emph{finiteness of the speed of propagation}.
\end{rem}

Now we can state the energy estimate with the optimal regularity assumption on the section $u$:

\begin{cor}\label{cor:EnEstOptimal}
The energy estimate \eqref{eq:EnEst} holds for all $t_0,t_1\in[T_0,T_1]$ with $t_0<t_1$ and for all $u\in \FE^{k}(M,P,t)$ with $\supp(u)\subset J(K)$.
\end{cor}

\begin{proof}
By \tref{thm:enest}, \eqref{eq:EnEst} holds for all $u\in C^\infty_{sc}(M)$ with $\supp(u)\subset J(K)$.
By \cref{cor:dense}, $C^\infty_{sc}(M)$ is dense in $\FE^k(M,P,t)$.
All terms in \eqref{eq:EnEst} are continuous with respect to the topology of $\FE^k(M,P,t)$.
Hence \eqref{eq:EnEst} holds for all $u\in \FE^{k}(M,P,t)$ with $\supp(u)\subset J(K)$.
\end{proof}

The space of finite $k$-energy solutions is independent of the choice of Cauchy temporal function.

\begin{cor}\label{cor:FEkerPIndep}
Let $t$ and $\tt$ be Cauchy temporal functions on $M$.
Then for every $k\in\R$
\[
\FE^k(M,\ker(P),t) = \FE^k(M,\ker(P),\tt) .
\]
\end{cor}

\begin{proof}
a)
Let $\Sigma_\bullet$ and $\tS_\bullet$ be the foliations by Cauchy hypersurfaces corresponding to $t$ and to $\tt$, respectively.
We assume first that $\Sigma_\bullet$ and $\tS_\bullet$ have one Cauchy hypersurface $\Sigma_0$ in common.
By \cref{cor:CauchyUnique}, finite $k$-energy solutions to the equation $Pu=0$ are uniquely determined by their Cauchy data.
Thus the composition of the isomorphisms, both given by $u\mapsto (u|_{\Sigma_0},\na_\nu u|_{\Sigma_0})$,
\[
\FE^k(M,\ker(P),t)
\xrightarrow{\cong}
H^k_c(\Sigma_0) \oplus H^{k-1}_c(\Sigma_0)
\xleftarrow{\cong}
\FE^k(M,\ker(P),\tt)
\]
must be the identity map, at least for $u\in C^\infty_{sc}(M) \cap \ker(P)$.
Since $C^\infty_{sc}(M) \cap \ker(P)$ is dense in both $\FE^k(M,\ker(P),t)$ and $\FE^k(M,\ker(P),\tt)$ by \cref{cor:dense}, the assertion follows in case of a common Cauchy hypersurface.

b)
Next we drop the assumption that the Cauchy temporal functions have a Cauchy hypersurface in common, but we assume that there are Cauchy hypersurfaces $\Sigma_\tau$ and $\tS_{\tilde\tau}$ which are disjoint.
W.l.o.g.\ we assume that $\Sigma_\tau$ lies in the past of $\tS_{\tilde\tau}$.

In \cite{BS06} it is shown that given a smooth spacelike Cauchy hypersurface, one can find a Cauchy temporal function such that the given Cauchy hypersurface appears as a level set of the function (Theorem~1.2.B).
A minor modification of the proof also shows that one can prescribe two disjoint Cauchy hypersurfaces as level sets.
Hence there exists a Cauchy temporal function $\hat{t}:M\to\R$ such that $\widehat\Sigma_0=\Sigma_\tau$ and $\widehat\Sigma_1=\tS_{\tilde\tau}$.
Applying part a) twice we get
\[
\FE^k(M,\ker(P),t) = \FE^k(M,\ker(P),\hat{t}) = \FE^k(M,\ker(P),\tt) \,.
\]

c)
Now we drop all additional assumptions on $t$ and $\tt$.
We fix Cauchy hypersurfaces $\Sigma_\tau$ and $\tS_{\tilde\tau}$.
Since $I^+(\Sigma_\tau) \cap I^+(\tS_{\tilde\tau})$ is globally hyperbolic, it contains a smooth spacelike Cauchy hypersurface $\check\Sigma$.
One easily sees that $\check\Sigma$ is also a Cauchy hypersurface for $M$.
\begin{center}
\begin{pspicture}(-4,-2)(4,2)
\psline[linecolor=lightgray,fillcolor=lightgray,fillstyle=solid](-2,-2)(2,-2)(2,2)(-2,2)
\psecurve[linewidth=.5mm](-3,-1.5)(-2,-1.1)(0,-1)(2,-.6)(3,-1)
\psecurve[linewidth=.5mm](-3,-1)(-2,-.7)(0,-1)(2,-1.6)(3,-1.2)
\psecurve[linewidth=.5mm](-3,.7)(-2,.5)(0,1)(2,.9)(3,1)

\uput{0}[0](1,1){\psframebox*[framearc=.3]{$\check\Sigma$}}
\uput{0}[0](-1.7,-0.6){\psframebox*[framearc=.3]{$\tS_{\tilde\tau}$}}
\uput{0}[0](1,-.8){\psframebox*[framearc=.3]{$\Sigma_\tau$}}
\end{pspicture}

\emph{Fig.~2}
\end{center}
Applying \cite[Thm.~1.2.B]{BS06} we find a Cauchy temporal function $\check{t}:M\to\R$ possessing $\check\Sigma$ as a level set.
Since $\Sigma_\tau$ and $\check\Sigma$ are disjoint as well as $\tS_{\tilde\tau}$ and $\check\Sigma$, we can apply part b) twice and we get
\[
\FE^k(M,\ker(P),t) = \FE^k(M,\ker(P),\check{t}) = \FE^k(M,\ker(P),\tt) \,. \qedhere
\]
\end{proof}

We have seen that, in general, the space $\FE^k(M,P,t)$ does depend on the choice of Cauchy temporal function $t$ but there is an important exception:

\begin{cor}\label{cor:FEPIndep}
Let $t$ and $\tt$ be Cauchy temporal functions on $M$.
Then
\[
\FE^1(M,P,t) = \FE^1(M,P,\tt) .
\]
\end{cor}

\begin{proof}
We recall from Subsection~\ref{sss:DepTime} that $L^2\locsc(t(M),H^0(\Sigma_\bullet))$ is independent of $t$.
Using the isomorphism 
\begin{align*}
\FE^1(M,P,t)
\to
H^1_c(\Sigma_\tau) &\oplus H^{0}_c(\Sigma_\tau) \oplus L^2\locsc(t(M),H^{0}(\Sigma_\bullet)) \, , \\
u &\mapsto (u|_{\Sigma_0},\na_\nu u|_{\Sigma_0},Pu) \, ,
\end{align*}
the same reasoning as in the proof of \cref{cor:FEkerPIndep} yields the claim.
\end{proof}

Thus there is no need to keep the Cauchy temporal function $t$ in the notation for finite $k$-energy \emph{solutions}.
From now on we will briefly write $\FE^k(M,\ker(P))$ instead of $\FE^k(M,\ker(P),t)$ and $\FE^1(M,P)$ instead of $\FE^1(M,P,t)$.

\begin{cor}\label{cor:FEHk}
For each $k\in\N$ we have the continuous embedding
\begin{equation}
\FE^k(M,\ker(P))
\hookrightarrow
H^k\loc(M).
\label{eq:Hkreg}
\end{equation}
Moreover, we have a continuous embedding
\begin{equation}
\FE^1(M,P)
\hookrightarrow
H^1\loc(M).
\label{eq:H1reg}
\end{equation}
\end{cor}

\begin{proof}
Let $K\subset M$ be compact.
Then we can choose $[T_0,T_1]\subset t(M)$ such that $K\subset \Sigma[T_0,T_1]$.
Let $K'\subset\Sigma_{T_0}$ be compact.
\cref{cor:Hkest} yields the estimate
\begin{align*}
\|u\|^2_{H^k(K)}
&\leq
\|u\|^2_{H^k(\Sigma[T_0,T_1])} \\
&\leq
C_1\cdot\left(\|u|_{\Sigma_{T_0}}\|^2_{H^k(\Sigma_{T_0})}
+ \|\na_tu\|^2_{H^{k-1}(\Sigma_{T_0})}\right) \\
&\le
C_1\cdot\left(\|u\|^2_{I,J(K'),0,k} +  \|u\|^2_{I,J(K'),1,k-1}\right)
\end{align*}
for all $u\in C^\infty_{J(K')}(M)\cap \ker(P)$.
Here $I\subset t(M)$ is any compact subinterval which contains $T_0$.
Since $C^\infty_{sc}(M)\cap \ker(P)$ is dense in $\FE^k(M,\ker(P))$ by \cref{cor:dense}, the same estimate holds for all $u\in C^0_{J(K')}(t(M),H^k(\Sigma_\bullet)) \cap C^1_{J(K')}(t(M),H^{k-1}(\Sigma_\bullet)) \cap \ker(P)$.
This shows that we have a continuous embedding
\[
C^0_{J(K')}(t(M),H^k(\Sigma_\bullet)) \cap C^1_{J(K')}(t(M),H^{k-1}(\Sigma_\bullet)) \cap \ker(P)
\hookrightarrow
H^k\loc(M)
\]
for every compact subset $K'\subset \Sigma_{T_0}$.
Since every spatially compact subset of $M$ is contained in $J(K')$ for some compact subset $K'\subset \Sigma_{T_0}$, we get a continuous embedding as in \eqref{eq:Hkreg}.

The embedding $\FE^1(M,P)\hookrightarrow H^1\loc(M)$ is obtained similarly using the estimate $\|u\|^2_{H^1(\Sigma[t_0,t_1])}
\leq C\cdot\left(\|u|_{\Sigma_{t_0}}\|^2_{H^1(\Sigma_{t_0})} + \|\na_tu\|^2_{L^2(\Sigma_{t_0})} + \|Pu\|^2_{L^2(\Sigma[t_0,t_1])}\right)$.
\end{proof}

\section{The Goursat problem}

The Goursat problem is an initial value problem with initial data prescribed on a characteristic Cauchy hypersurface $\Sigma$.
For wave operators, ``characteristic'' means lightlike, i.e.\ the Lorentzian metric induces a degenerate metric on $\Sigma$.
A typical example would be the light cone $\Sigma=\partial J^+(x)$ for any point $x\in M$.
This $\Sigma$ is no longer smooth and we will not assume smoothness in the sequel.

Now there is an issue defining the right function spaces on $\Sigma$.
If $\Sigma$ is not smooth, $C^k$-sections are no longer defined for $k>0$.
Moreover, if the induced metric is degenerate, then the induced volume density vanishes.
This makes it difficult to define Sobolev spaces.

We solve these difficulties as follows:
since we expect the solutions of the initial value problem to be finite energy sections, we simply take the restrictions of those as admissible initial values.

\begin{dfn}
Let $\Sigma\subset M$ be a Lipschitz hypersurface.
Then we put
\[
H^1_c(\Sigma;S) := H^1_c(\Sigma) := \{u|_\Sigma \mid u\in \FE^1(M,P)\} \,.
\]
\end{dfn}

Note that by \cref{cor:FEHk}, $\FE^1(M,P) \subset H^1\loc(M)$.
By the trace theorem for Lipschitz hypersurfaces \cite[Thm.~1]{D96}, the restriction $u|_\Sigma$ is well defined, at least as an $H^{1/2}\loc$-section.
It is convenient here that $\FE^1(M,P)$ does not depend on the choice of a Cauchy temporal function by \cref{cor:FEPIndep}.
If $\Sigma$ is smooth and spacelike, this yields the same space as the Sobolev space $H^1_c(\Sigma)$ defined in Subsection~\ref{sss:SobCompSupp}.

\subsection{A general existence result}
To prepare for the existence part in the Goursat problem we first observe the following general existence theorem which does not yet make any reference to the initial value surface being characteristic.

\begin{thm}\label{thm:GeneralExistence}
Let $\Sigma\subset M$ be any Cauchy hypersurface.
For any $f\in L^2\locsc(M)$ and any $u_0\in H^1_c(\Sigma)$ there exists $u\in\FE^1(M,P)$ such that $Pu=f$ and $u|_\Sigma=u_0$.
\end{thm}

\begin{proof}
a)
The past $I^-(\Sigma)$ of $\Sigma$ is globally hyperbolic and hence has a smooth spacelike Cauchy hypersurface $\Sigma_0$.
One checks that $\Sigma_0$ is also a Cauchy hypersurface for $M$.
\begin{center}
\begin{pspicture}(-4,-2)(4,2)
\psline[linecolor=lightgray,fillcolor=lightgray,fillstyle=solid](-2,-2)(2,-2)(2,2)(-2,2)
\psline[linecolor=gray,fillcolor=gray,fillstyle=solid](-2,2)(0,0)(2,2)
\psline[linewidth=.5mm](-2,2)(0,0)(2,2)
\psecurve[linewidth=.5mm](-3,-1.5)(-2,-1.1)(0,-1)(2,-.6)(3,-1)

\uput{0}[0](1,1){\psframebox*[framearc=.3]{$\Sigma$}}
\uput{0}[0](-1.7,-0.3){\psframebox*[framearc=.3]{$I^-(\Sigma)$}}
\uput{0}[0](1,-.8){\psframebox*[framearc=.3]{$\Sigma_0$}}
\end{pspicture}

\emph{Fig.~3}
\end{center}
Let $\chi_+\in L^\infty(M,\R)$ be the characteristic function of $J^+(\Sigma)$.
Then $\chi_+\cdot f\in L^2\locsc(M)=L^2\locsc(t(M),H^0(\Sigma_\bullet))$.
By \tref{thm:Cauchy1}, there is a unique $u_+\in\FE^1(M,P)$ such that $Pu_+=\chi_+\cdot f$ and $u_+|_{\Sigma_0} = \na_\nu u_+|_{\Sigma_0} = 0$.
Moreover, $\supp(u_+) \subset J^+(\supp(\chi_+f)) \subset J^+(J^+(\Sigma)) = J^+(\Sigma)$.
Since $u\equiv 0$ on $I^-(\Sigma)$, we have $u_+|_\Sigma=0$.

Now let $\chi_-\in L^\infty(M,\R)$ be the characteristic function of $J^-(\Sigma)$.
Replacing $\Sigma_0$ by a smooth spacelike Cauchy hypersurface in the future of $\Sigma$, the same arguments yield a section $u_-\in\FE^1(M,P)$ with $Pu_-=\chi_-\cdot f$ and  $u_-|_\Sigma=0$.
For $\tilde u:=u_++u_-\in\FE^1(M,P)$ we have $P\tilde u=f$ and $\tilde u|_\Sigma=0$.

b)
Let $w\in\FE^1(M,P)$ be such that $w|_\Sigma=u_0$.
We apply part a) of the proof with $f$ replaced by $-Pw$.
This yields $v\in\FE^1(M,P)$ with $Pv=-Pw$ and $v|_\Sigma=0$.
Then $u:=\tilde u+w+v\in\FE^1(M,P)$ satisfies $Pu=f+Pw-Pw=f$ and $u|_\Sigma=w|_\Sigma=u_0$.
\end{proof}

\subsection{The characteristic initial value problem}
The previous existence statement does not require any assumption on the Cauchy hypersurface, neither on its regularity nor on its causal type.
Uniqueness cannot be expected in this generality because we know from the discussion of the Cauchy problem that, in the spacelike case, we also need to prescribe the normal derivative along the Cauchy hypersurface in order to uniquely determine the solution.

In the characteristic case, the situation is different.
Let us first make this more precise.
Any partial Cauchy hypersurface $\Sigma$ is Lipschitz and hence has a tangent space at almost all points due to Rademacher's theorem.
We call $\Sigma$ \emph{characteristic} if the induced metric degenerates on these tangent spaces.
Now we have:

\begin{thm}\label{thm:CharacteristicUniqueness}
Let $\Sigma\subset M$ be a characteristic partial Cauchy hypersurface.
Assume that $J^+(\Sigma)$ is past compact.

Then for any $f\in L^2\locsc(M)$ and any $u_0\in H^1_c(\Sigma)$ there exists $u\in\FE^1(M,P)$ such that $Pu=f$ on $J^+(\Sigma)$ and $u|_\Sigma=u_0$.
On $J^+(\Sigma)$, $u$ is unique.
%
\end{thm}

\begin{proof}
a)
Since the past $I^-(\Sigma)$ is globally hyperbolic it contains a smooth spacelike Cauchy hypersurface $\Sigma_0$.
We need to check that $\Sigma_0$ is also a Cauchy hypersurface for $\tilde M:=I^+(\Sigma)\dot\cup\Sigma\dot\cup I^-(\Sigma)\subset M$.
This is not obvious anymore because $\Sigma$ is only a partial Cauchy hypersurface.

Let $c:I\to \tilde M$ be an inextensible timelike curve.
W.l.o.g.\ we assume that $I=(-1,1)$ and that $c$ is future-directed.
We put $t_0:=\sup\{t\in(-1,1)\mid c((-1,t))\subset I^-(\Sigma)\}$.

\emph{Case 1:} $t_0=1$.

Then $c$ is an inextensible timelike curve in $I^-(\Sigma)$ and hence intersects $\Sigma_0$ exactly once.

\emph{Case 2:} $-1<t_0<1$.

Now $c|_{(-1,t_0)}$ is an inextensible timelike curve in $I^-(\Sigma)$ and hence intersects $\Sigma_0$ exactly once.
Moreover, $c|_{[t_0,1)}$ is contained in $J^+(\Sigma)$ and therefore does not intersect $\Sigma_0$.
Altogether, $c$ intersects $\Sigma_0$ exactly once.

\emph{Case 3:} $t_0=-1$.

We show that this case cannot occur.
It would mean that $c$ is entirely contained in $J^+(\Sigma)$.
Thus $c((-1,0])$ would be contained in $J^+(\Sigma)\cap J^-(c(0))$ which is a compact set by past-compactness of $J^+(\Sigma)$.
Hence $c(s)$ has an accumulation point as $s\searrow-1$.
Since $c$ is timelike this accumulation point is unique.
Thus one can extend $c$ continuously to the past which contradicts inextensibility of $c$.

This proves that $\Sigma_0$ is indeed a Cauchy hypersurface for $\tilde M$.
Now proceeding as in the proof of \tref{thm:GeneralExistence} one shows existence of $u$.

b)
As to uniqueness, suppose $u|_\Sigma=0$ and $Pu=0$ holds on $J^+(\Sigma)$.
Let $\phi\in C^\infty_c(M;S^*)$ be a test section with $\supp(\phi)\subset J^+(\Sigma)$.
We need to show $u[\phi]=0$.

Let $G^\dagger_-:C^\infty_c(M;S^*)\to C^\infty(M;S^*)$ be the retarded Green's operator for the formally dual operator $P^\dagger$, see \cite[Sec.~3.4]{BGP07}.
Now $\supp(G^\dagger_-(\phi)) \subset J^-(\supp\phi)$.
Since $J^+(\Sigma)$ is past compact, the set $J^+(\Sigma) \cap J^-(\supp\phi)$ is compact.
The Green's formula~\eqref{eq:Green} with $\psi = G^\dagger_-(\phi)$ yields
\begin{align*}
u[\phi]
&=
\int_{J^+(\Sigma)} \phi(u)\, dV
=
\int_{J^+(\Sigma)} (P^\dagger G^\dagger_-(\phi))(u)\, dV \\
& =
\int_{J^+(\Sigma)} (G^\dagger_-(\phi))(Pu)\, dV 
+ \int_{\partial J^+(\Sigma)}  (G^\dagger_-(\phi)(\na_Lu) - \hat\na_L G^\dagger_-(\phi)(u))\cdot \AL \, .
\end{align*}
The first integral vanishes because $Pu=0$ on $J^+(\Sigma)$.
The boundary term vanishes because $u=0$ on $\partial J^+(\Sigma)$.
Note that $L$ is tangential to the boundary so that $\na_Lu=0$ as well.
Thus $u[\phi]=0$ for every test section with support in $J^+(\Sigma)$.
Hence $u=0$ on $J^+(\Sigma)$.
\end{proof}

The assumption that $J^+(\Sigma)$ be past compact is crucial.
If we drop it, \tref{thm:CharacteristicUniqueness} fails:

\begin{ex}
Let $M$ be the $(1+1)$-dimensional Minkowski space with standard coordinates $x_0,x_1$.
The boundaries of the future and of the past lightcone, $\Sigma'=\partial J^+(0,0)$ and $\Sigma=\partial J^-(0,0)$, are both characteristic partial Cauchy hypersurfaces.
Now $J^+(\Sigma')$ is past compact while $J^+(\Sigma)$ is not.
For instance, $J^-(1,0) \cap J^+(\Sigma)$ is not compact.
\begin{center}
\begin{pspicture}(-6.4,-2.8)(8,2.2)
\psline[linecolor=lightgray,fillcolor=lightgray,fillstyle=solid](-3.25,-2)(-6,0.75)(-6,2)(-0.5,2)(-0.5,0.75)
\psline[linecolor=gray,fillcolor=gray,fillstyle=solid](-3.25,1)(-1.75,-0.5)(-3.25,-2)(-4.75,-0.5)
\psline[linewidth=0.2mm](-4.75,-0.5)(-3.25,1)(-1.75,-0.5)
\psline[linewidth=.5mm](-6,0.75)(-3.25,-2)(-0.5,0.75)

\uput{0}[0](-1.9,1.5){\psframebox*[framearc=.3]{$J^+(\Sigma')$}}
\uput{0}[0](-4.7,0){\psframebox*[framearc=.3]{$J^-(1,0)\cap J^+(\Sigma')$}}
\uput*[315](-1.75,-0.5){$\Sigma'$}
\uput{0}[0](-5,-2.5){$J^+(\Sigma')$ is past compact}

\psline[linecolor=lightgray,fillcolor=lightgray,fillstyle=solid](6,-2)(6,2)(0.5,2)(0.5,-2)
\psline[linecolor=gray,fillcolor=gray,fillstyle=solid](3.25,1)(6,-1.75)(6,-2)(0.5,-2)(0.5,-1.75)
\psline[linewidth=0.2mm](0.5,-1.75)(3.25,1)(6,-1.75)
\psline[linecolor=white,fillcolor=white,fillstyle=solid](5.35,-2.1)(3.25,0)(1.15,-2.1)
\psline[linewidth=.5mm](5.25,-2)(3.25,0)(1.25,-2)

\uput{0}[0](4.5,1.5){\psframebox*[framearc=.3]{$J^+(\Sigma)$}}
\uput{0}[0](1.8,0.5){\psframebox*[framearc=.3]{$J^-(1,0)\cap J^+(\Sigma)$}}
\uput*[225](4.25,-1){$\Sigma$}
\uput{0}[0](1.3,-2.5){$J^+(\Sigma)$ is not past compact}
\end{pspicture}

\emph{Fig.~4}
\end{center}
Indeed, \tref{thm:CharacteristicUniqueness} holds for $\Sigma'$ but not for $\Sigma$.
Let $v\in C^\infty_c(\R,\R)$ be such that $\supp(v)=[1,2]$.
We put $u(x_0,x_1) := v(x_0-x_1)$.
Then $u\in C^\infty_{sc}(M,\R)$ solves the wave equation $\Box u=0$.
It is a ``right traveling wave''.
The support $\supp(u) = \{(x_0,x_1)\in M \mid x_0-2 \leq x_1 \leq x_0-1\}$ is spatially compact and does not meet $\Sigma$.
Hence $u|_\Sigma=0$ but $u\not\equiv 0$ on $J^+(\Sigma)$.
\begin{center}
\begin{pspicture}(-4,-2)(4,2.2)
\psline[linecolor=lightgray,fillcolor=lightgray,fillstyle=solid](-4,-2)(-2,-2)(0,0)(2,-2)(3,-2)(3,1)(-4,1)(-4,-2)
\psline[linewidth=.5mm](-2,-2)(0,0)(2,-2)
\psline[linecolor=gray,fillcolor=gray,fillstyle=solid](-3,-2)(0,1)(-1,1)(-4,-2)(-1,1)(-4,-2)
\psline[linewidth=0.2mm](-3,-2)(0,1)
\psline[linewidth=0.2mm](-4,-2)(-1,1)

\uput*[225](1,-1){$\Sigma$}
\uput{0}[0](1.2,0){\psframebox*[framearc=.3]{$J^+(\Sigma)$}}
\uput{0}[0](-2.2,0){\psframebox*[framearc=.3]{$\supp(u)$}}
\end{pspicture}

\emph{Fig.~5}
\end{center}
The same discussion applies if one replaces $\Sigma=\partial J^-(0,0)$ by the characteristic hyperplane $\Sigma=\{(s,s)\in M \mid s\in\R\}$.
\end{ex}

\section*{Appendix. Green's formula for lightlike boundary}

Let $P$ be a wave operator acting on sections of a vector bundle $S$ over $(\Omega,g)$, a Lorentzian manifold with lightlike Lipschitz boundary $\partial \Omega$.
There exists a connection $\na$ on $S$ and an endomorphism field $B$ on $S$ such that 
\[
P = \Box^\na + B
\]
where $\Box^\na$ is the connection-d'Alembert operator explained in Example~\ref{ex:connectiondAlembert}.
See e.g.\ \cite[Lem.~1.5.5]{BGP07} or \cite[Prop.~3.1]{BK96} for a proof.
Let $\hat\na$ be the induced connection on the dual bundle $S^*$.
It is characterized by $\partial_X(\psi(u))=(\hat\na_X\psi)(u) + \psi(\na_Xu)$ for all differentiable sections $u$ of $S$ and $\psi$ of $S^*$ and all tangent vectors $X$.

We now assume that $\Omega$ is oriented and let $\vol$ be the volume form.
This is no serious restriction because we can always pass to the orientation covering.
For any base $(b_0,b_1,\ldots,b_n)$ of $T_x\Omega$ we put $g_{ij}:=g(b_i,b_j)$ and let $(g^{ij})$ be the matrix inverse to $(g_{ij})$.
We fix $C^2$-sections $u$ of $S$ and $\psi$ of $S^*$.
Now the $n$-covector
\[
\eta|_x := \sum_{ij}g^{ij}\big(\hat\na_i\psi(u) - \psi(\na_iu)\big)\cdot b_j \lrcorner \vol \in \Lambda^nT^*_x\Omega
\]
is defined independently of the choice of base. 
Here $b_j \lrcorner \vol$ denotes the insertion of $b_j$ into the first slot of $\vol$, i.e.\ $b_j \lrcorner \vol = \vol(b_j,\cdot,\ldots,\cdot)$.
Now $\eta$ is a globally defined $n$-form on $\Omega$ of $C^1$-regularity.

To compute the exterior derivative of $\eta$ we may assume that the tangent frame $(b_0,b_1,\ldots,b_n)$ is chosen synchronous at the point $x$ under consideration, i.e.\ $\na b_j=0$ at $x$.
Then the derivatives of $g^{ij}$ also vanish at $x$ and we get at $x$
\begin{align*}
d\eta
&=
\sum_{ijk}g^{ij}\partial_k\big(\hat\na_i\psi(u) - \psi(\na_iu)\big)\cdot b_k^*\wedge (b_j \lrcorner \vol) \\
&=
\sum_{ij}g^{ij}\partial_j\big(\hat\na_i\psi(u) - \psi(\na_iu)\big)\cdot \vol \\
&=
\sum_{ij}g^{ij}\big(\hat\na_j\hat\na_i\psi(u) + \hat\na_i\psi(\na_ju)
- \hat\na_j\psi(\na_iu) - \psi(\na_j\na_iu)\big)\cdot \vol \\
&=
\sum_{ij}g^{ij}\big(\hat\na_j\hat\na_i\psi(u) - \psi(\na_j\na_iu)\big)\cdot \vol \\
&=
\big((\Box^{\hat\na}\psi)(u) - \psi(\Box^\na u)\big)\cdot \vol \, .
\end{align*}
If $\supp(\psi) \cap \supp(u)$ is compact and contained in the interior of $\Omega$, then 
\[
0 = \int_\Omega d\eta = \int_\Omega \big((\Box^{\hat\na}\psi)(u) - \psi(\Box^\na u)\big)\, dV \, .
\]
Thus $\Box^{\hat\na}$ is the formal dual of $\Box^{\na}$.
We conclude
\[
P^\dagger = \Box^{\hat\na} + B^\dagger
\]
where $B^\dagger$ is the pointwise adjoint endomorphism field of $S^*$.
We are interested in the boundary term which occurs if $\supp(\psi) \cap \supp(u)$ is no longer contained in the interior of $\Omega$.

Let $x\in\partial \Omega$ be a point at which the boundary is differentiable.
Then there is a lightlike vector $L\in T_x\partial \Omega$, unique up to multiples.
We choose a lightlike vector $\check L\in T_x\Omega$ such 
\begin{equation}
g(L,\check L)=-1.
\label{eq:LLcheck}
\end{equation}
Although $\check L$ is not uniquely determined by $L$, the restriction of the $n$-covector $\check L \lrcorner \vol$ to $T_x\partial \Omega$ is determined by $L$.
We denote it by $\AL := \check L \lrcorner \vol \in \Lambda^nT_x^*\partial \Omega$.
Since $\check L$ is not tangent to $\partial \Omega$, the $n$-covector $\AL$ is nonzero.
If we replace $L$ by a multiple $\alpha L$, $\alpha\in\R\setminus\{0\}$, then we may simply replace $\check L$ by $\tfrac{1}{\alpha}\check L$ in order to keep \eqref{eq:LLcheck} valid.
Hence $\mathrm{A}_{\alpha L} = \tfrac{1}{\alpha}\AL$.

To identify the boundary term we express $\eta$ at a regular point of $\partial \Omega$ using a base of the form $\check L,L,b_2,\ldots,b_n$ where $L$ is lightlike and tangential to the boundary and $b_2,\ldots,b_n$ are spacelike and tangential to the boundary.
\begin{center}
\psset{unit=0.25Em}
\begin{pspicture}(60,50)
\pscustom[linewidth=0.35]{\newpath \psline(0.1,0)(25.1,40)}
\pscustom[linewidth=0.35]{\newpath \psbezier(25.1,40)(52.4,40.6)(48.1,48.5)(48.5,50.1)}
\pscustom[linewidth=0.35]{\newpath \psline(48.1,46.3)(23.4,6.9)}
\pscustom[linewidth=0.35,strokeopacity=0.7]{\newpath \psline(48.5,50)(43.6,42.7)} \uput{1}[0](48.5,50){$\partial\Omega$}

\qdisk(23.5,20.4){0.5} \uput{1.2}[265](23.5,20.4){$x$}
\pscustom[linewidth=0.25]{\newpath \psline{->}(23.5,20.4)(30.7,31.9)} \uput{1}[0](30.7,31.9){$L$} 
\pscustom[linewidth=0.25]{\newpath \psline{->}(23.5,20.4)(36.8,22.8)} \uput{1}[0](36.8,22.8){$b_2,\ldots,b_n$}

\pscustom[linewidth=0.25,linestyle=dashed,dash=0.3 0.3]{\newpath \psline(23.5,20.4)(16.2,25.8)}
\pscustom[linewidth=0.25]{\newpath \psline{->}(16.2,25.8)(14.3,27.2)} \uput{1}[180](14.3,27.2){$\check{L}$}
\end{pspicture}

\emph{Fig.~6}
\end{center}

Then the pull-back of $\eta$ to $\partial \Omega$ takes the form
\[
\eta = -(\hat\na_L\psi(u) - \psi(\na_Lu))\cdot \AL \, .
\]
Due to the scaling property of $\AL$, this expression is independent of the choice of $L$.
The Stokes' theorem for manifolds with Lipschitz boundary (see e.g.\ \cite[p.~282]{Alt}) yields
\[
\int_{\partial \Omega} \eta
=
\int_\Omega d\eta 
= 
\int_\Omega \big((\Box^{\hat\na}\psi)(u) - \psi(\Box^\na u)\big)\, dV 
=
\int_\Omega \big((P^\dagger\psi)(u) - \psi(Pu)\big)\, dV \, .
\]

\begin{lem}[Green's formula]\label{lem:Green}
Let $\Sigma\subset M$ be a characteristic Lipschitz hypersurface.
Let $u\in\FE^1(M,P;S)$ and $\psi\in C^\infty(M;S^*)$ such that $\supp(u) \cap \supp\psi \cap J^+(\Sigma)$ is compact.
Then
\begin{equation}
\int_{J^+(\Sigma)} \big(\psi(Pu) - (P^\dagger\psi)(u)\big) dV
= 
\int_\Sigma (\hat\na_L\psi(u) - \psi(\na_Lu))\AL \, .
\label{eq:Green}
\end{equation}
\end{lem}

\begin{proof}
The previous considerations with $\Omega=J^+(\Sigma)$ prove the formula if $u\in C^2(M)$.
We fix $\psi$ and regard both the left hand side and the right hand side of \eqref{eq:Green} as linear functionals of $u$.
The left hand side is clearly continuous in $u$ with respect to the topology of $\FE^1(M,P)$.
As to the right hand side, the map $\FE^1(M,P) \hookrightarrow H^1\loc(M) \to H^{1/2}\loc(\Sigma)$, $u\mapsto u|_\Sigma$, is continuous.
This uses the trace theorem for Lipschitz boundaries, see \cite[Thm.~1]{D96}.
Therefore $\int_\Sigma \hat\na_L\psi(u)\AL$ is continuous in $u$ with respect to the topology of $\FE^1(M,P)$.
Moreover, $\na_L$ yields a continuous linear map $H^{1/2}\loc(\Sigma) \to H^{-1/2}\loc(\Sigma)$.
Pairing against a smooth compactly supported section $\psi$ is continuous on $H^{-1/2}\loc(\Sigma)$.
Therefore $\int_\Sigma \psi(\na_Lu)\AL$ is also continuous in $u$ with respect to the topology of $\FE^1(M,P)$.

Thus the right hand side is continuous in $u$.
Since smooth sections are dense in $\FE^1(M,P)$ by \cref{cor:dense}, both sides have to agree for $u\in\FE^1(M,P)$.
\end{proof}


\end{document}